\documentclass[11pt]{amsart}
\usepackage{amsmath,amssymb,amsthm,mathrsfs,enumerate,bm,xcolor,multirow,pbox,latexsym}
\usepackage{graphicx,color,framed,tikz}
\allowdisplaybreaks[4]
\numberwithin{equation}{section}
\newcommand{\N}{\mathbb{N}}
\newcommand{\R}{\mathbb{R}}

\newcommand{\E}{\mathbb{E}}
\newcommand{\Prob}{\mathbb{P}}
\newcommand{\G}{\mathbb{G}}

\newcommand{\pnorm}[2]{\lVert#1\rVert_{#2}}

\newcommand{\biggpnorm}[2]{\bigg\lVert#1\bigg\rVert_{#2}}
\newcommand{\abs}[1]{\lvert#1\rvert}
\newcommand{\bigabs}[1]{\big\lvert#1\big\rvert}
\newcommand{\biggabs}[1]{\bigg\lvert#1\bigg\rvert}

\renewcommand{\epsilon}{\varepsilon}

\renewcommand{\d}[1]{\mathrm{d}#1}

\newcommand{\beq}{\begin{equation}}
\newcommand{\eeq}{\end{equation}}
\newcommand{\beqa}{\begin{equation} \begin{aligned}}
\newcommand{\eeqa}{\end{aligned} \end{equation}}
\newcommand{\beqas}{\begin{equation*} \begin{aligned}}
\newcommand{\eeqas}{\end{aligned} \end{equation*}}

\newcommand{\bit}{\begin{itemize}}
	\newcommand{\eit}{\end{itemize}}
\newcommand{\bmat}{\begin{bmatrix}}
	\newcommand{\emat}{\end{bmatrix}}

\AtBeginDocument{%
	\def\MR#1{}
}

\theoremstyle{definition}\newtheorem{problem}{Problem}[section]
\theoremstyle{definition}\newtheorem{definition}[problem]{Definition}
\theoremstyle{remark}\newtheorem{assumption}{Assumption}

\theoremstyle{remark}\newtheorem{remark}[problem]{Remark}
\theoremstyle{definition}\newtheorem{example}[problem]{Example}
\theoremstyle{plain}\newtheorem{theorem}[problem]{Theorem}
\theoremstyle{plain}\newtheorem{question}{Question}
\theoremstyle{plain}\newtheorem{lemma}[problem]{Lemma}
\theoremstyle{plain}\newtheorem{proposition}[problem]{Proposition}
\theoremstyle{plain}\newtheorem{corollary}[problem]{Corollary}
\theoremstyle{plain}

\begin{document}

\title[Uniform limit theorems in complex sampling designs]{Complex Sampling Designs: Uniform Limit Theorems and Applications}
\thanks{The research of J. A. Wellner is partially supported by NSF Grant DMS-1566514, NI-AID grant 2R01 AI291968-04, a Simons Fellowship via the Newton Institute (INI-program STS 2018), Cambridge University,and the Saw Swee Hock Visiting Professorship of Statistics at the National University of Singapore (in 2019). }

\author[Q. Han]{Qiyang Han}

\address[Q. Han]{
Department of Statistics, Rutgers University, Piscataway, NJ 08854, USA.
}
\email{qh85@stat.rutgers.edu}

\author[J. A. Wellner]{Jon A. Wellner}

\address[J. A. Wellner]{
Department of Statistics, Box 354322, University of Washington, Seattle, WA 98195-4322, USA.
}
\email{jaw@stat.washington.edu}

\date{\today}

\maketitle

\begin{abstract}
In this paper, we develop a general approach to proving global and local uniform limit theorems for the Horvitz-Thompson empirical process arising from complex sampling designs. Global theorems such as Glivenko-Cantelli and Donsker theorems, and local theorems such as local asymptotic modulus and related ratio-type limit theorems are proved for both the Horvitz-Thompson empirical process, and its calibrated version. Limit theorems of other variants and their conditional versions are also established. Our approach reveals an interesting feature: the problem of deriving uniform limit theorems for the Horvitz-Thompson empirical process is essentially no harder than the problem of establishing the corresponding finite-dimensional limit theorems. These global and local uniform limit theorems are then applied to important statistical problems including (i) $M$-estimation (ii) $Z$-estimation (iii) frequentist theory of Bayes procedures, all with weighted likelihood, to illustrate their wide applicability.
\end{abstract}

%\tableofcontents

\section{Introduction}

\subsection{Overview}

Over the past thirty years, uniform limit theorems for the empirical process have proved to be a universal tool in various statistical problems based on independent observations; we only refer readers to the textbooks  \cite{gine2015mathematical,kosorok2008intro,van2000empirical,van1996weak} for relevant theoretical developments and various statistical applications. 

Our focus here will be uniform limit theorems for the Horvitz-Thompson empirical process arising from complex sampling designs (cf. \cite{sarndal1992model}). Such limit theorems provide fundamental probabilistic tools in statistical applications with survey data, for instance, in combination with the functional delta method (see e.g. \cite{barrett2009statistical,bhattacharya2007inference,bhattacharya2011nonparametric,davidson2009reliable} for applications in econometrics), or in semi-parametric modeling (see e.g. \cite{breslow2003large,breslow2009improved,breslow2009using,lin2000fitting,nan2009asymptotic,nan2013general} for applications in biostatistics), just to name a few. Recent years have seen the emergence of interest in further limit theory in this direction (e.g. \cite{bertail2017empirical,boistard2017functional,breslow2007weighted,breslow2008theorem,conti2014estimation,saegusa2018large,saegusa2013weighted}), but the scope of the existing results in this direction has been somewhat limited, and many of these available results have been derived based on case-by-case analyses. Roughly speaking, there are three approaches so far in the literature:
\begin{enumerate}
	\item  \cite{breslow2007weighted,breslow2008theorem} developed theory in the context of two-phase sampling with phase II a simple sampling without replacement sampling design. The key idea therein is to view the Horvitz-Thompson empirical process conditionally as an exchangeably weighted bootstrap empirical process \cite{praestgaard1993exchangeably}. This idea is further exploited in \cite{saegusa2013weighted} in the context of calibrated Horvitz-Thompson empirical processes. A similar bootstrap approach is adopted in \cite{saegusa2018large} in the setting of stratified sampling with potential overlaps. 
	\item \cite{bertail2017empirical} derived a Donsker theorem for the Bernoulli sampling design and other sampling designs that are close enough to the rejective sampling design (= high entropy designs) under a uniform entropy condition on the indexing function class. Their techniques heavily rely on the conditional independence of the inclusion indicators. 
	\item 
%	\cite{conti2014estimation} established a Donsker theorem over one particular class $\{\bm{1}_{(-\infty,t]}:t \in \R\}$ under high entropy designs, by explicit calculations that verify the one-dimensional tightness condition. This line of work was further pursued by \cite{boistard2017functional} who proved the same type of Donsker theorem over $\R$ under general sampling designs.
\cite{conti2014estimation} and \cite{boistard2017functional} established Donsker theorems over one class $\{\bm{1}_{(-\infty,t]}:t \in \R\}$ under sampling designs with increasing level of generality, by explicit calculations that verify the one-dimensional tightness condition.
\end{enumerate}

The apparent case-by-case complication here is that complex sampling designs typically induce complicated dependence structure between the samples, so in order to use existing techniques from empirical process theory, certain latent independence or exchangeability structure needs to be identified in a case-by-case routine. 

On the other hand, some structural commonality is indeed hinted at by the results proved in the above cited papers: uniform laws of large numbers (i.e. Glivenko-Cantelli theorems) and uniform central limit theorems (i.e. Donsker theorems) hold under rather minimal conditions on the indexing function classes. The intriguing question naturally arises:

\begin{question}\label{question}
Does there exist any general approach to proving uniform limit theorems for the Horvitz-Thompson empirical process under natural conditions, without being confined to a particular form of the sampling design?
\end{question} 

A possible solution to this very natural question, however, appears far from obvious from the previously 
%aforementioned 
described approaches. The challenges involved here were already noted in Lin \cite{lin2000fitting} as ``......\emph{To our knowledge there does not exist a general theory on conditions required for the tightness and weak convergence of Horvitz-Thompson processes}......'', dating back to as early as 2000. One of the goals of this paper is to address Question \ref{question} in an appropriate general framework that includes a wide variety of sampling designs. Part of the philosophical difficulty in such a general approach is that there is an easily believable impression that any general attempt at establishing global uniform limit theorems for the Horvitz-Thompson empirical process, must necessarily give general recipes for establishing finite-dimensional convergence of the Horvitz-Thompson empirical process. In the specific context of Donsker theorems, this impression pushes one to think about the `right conditions' under which at least central limit theorems hold for a single function under various different sampling designs---a task that usually already requires a case-by-case study.

In this paper, we show that this easily believable impression need not be the rule in the context of uniform limit theorems for Horvitz-Thompson empirical processes, at least in the super-population framework adopted in \cite{boistard2017functional,rubin2005two} with uniformly positive first-order inclusion probabilities.
%has been an attempt in giving a unified treatment of deriving (global) uniform limit theorems for weighted empirical processes arise in the complex survey sampling literature. 
The major `change of thinking' adopted in the current paper, interestingly, indicates that \emph{the problem of deriving uniform limit theorems for Horvitz-Thompson empirical processes is not really more difficult than that of establishing the corresponding finite-dimensional limit theorems}. In the context of Donsker theorems, this amounts to saying that, as long as the Horvitz-Thompson empirical process converges finite-dimensionally, weak convergence at the process level follows almost automatically. Since finite-dimensional convergence is necessary for weak convergence of the process to hold, the real point here is to separate the problem of establishing finite-dimensional convergence of the Horvitz-Thompson empirical process from that of establishing a uniform limit theorem. The approach here is in part inspired by a multiplier inequality developed in a recent work of the authors \cite{han2017sharp}, which holds regardless of the dependence structure among the multipliers, given sufficient independence structure between the multipliers and the samples. 

Establishing global uniform limit theorems serves as a first step in understanding the behavior of these Horvitz-Thompson empirical processes. In typical semi-/non-parametric applications, it is also of crucial importance to understand the local behavior of these empirical processes. To this end, we further study the local behavior of the Horvitz-Thompson empirical process by characterizing its local asymptotic modulus and proving several ratio-type limit theorems. These local uniform limit theorems show that the  Horvitz-Thompson empirical process typically has similar local behavior compared to its empirical process counterpart. Similar global and local uniform limit theorems are established for the calibrated version of the Horvitz-Thompson empirical processes. Some other variants of Horvitz-Thompson empirical processes are discussed. Conditional versions of the uniform limit theorems are also established.

As an illustration  and a proof of concept of the power of our global and local uniform limit theorems (and related techniques), we apply these new tools to a variety of important statistical problems, including (i) $M$-estimation, or \emph{empirical risk minimization}, in a general non-parametric model, (ii) $Z$-estimation in a general semi-parametric model, and (iii) frequentist theory of Bayesian procedures (i.e. theory of posterior contraction rates and Bernstein-von Mises type theorems), all based on weighted likelihood. Several concrete examples are illustrated to further demonstrate the applicability of these general results.

The rest of the paper is organized as follows. Section \ref{section:sampling_designs} is devoted to a general probabilistic framework for complex sampling designs and detailed illustrations of the theory in the context of a number of examples. Section \ref{section:theory} studies the global and local uniform limit theorems for the Horvitz-Thompson empirical process. Section \ref{section:applications} gives applications of the theory developed in Section \ref{section:theory} to the aforementioned statistical problems. Proofs are collected in Sections \ref{section:proof_theory}-\ref{section:remaining_proof}.

%\subsection{Overview}

\subsection{Notation}\label{section:notation}
For a real-valued measurable function $f$ defined on $(\mathcal{X},\mathcal{A},P)$ and $p\geq 1$, $\pnorm{f}{L_p(P)}\equiv \big(P\abs{f}^p)^{1/p}$ denotes the usual $L_p$-norm under $P$, and $\pnorm{f}{\infty}\equiv\pnorm{f}{L_\infty}\equiv  \sup_{x \in \mathcal{X}} \abs{f(x)}$. $f$ is said to be $P$-centered if $Pf=0$. $L_p(g,B)$ denotes the $L_p(P)$-ball centered at $g$ with radius $B$. For simplicity we write $L_p(B)\equiv L_p(0,B)$. 

Let $(\mathcal{F},\pnorm{\cdot}{})$ be a subset of the normed space of real functions $f:\mathcal{X}\to \R$. Let $\mathcal{N}(\epsilon,\mathcal{F},\pnorm{\cdot}{})$ be the $\epsilon$-covering number, and let  $\mathcal{N}_{[\,]}(\epsilon,\mathcal{F},\pnorm{\cdot}{})$ be the $\epsilon$-bracketing number; see page 83 of \cite{van1996weak} for more details. To avoid unnecessary measurability digressions, we assume that $\mathcal{F}$ is countable throughout the article. As usual, for any $\phi:\mathcal{F}\to \R$, we write $\pnorm{\phi(f)}{\mathcal{F}}$ for $ \sup_{f \in \mathcal{F}} \abs{\phi(f)}$. 

Throughout the article $\epsilon_1,\ldots,\epsilon_n$ will be i.i.d. Rademacher random variables independent of all other random variables. $C_{x}$ will denote a generic constant that depends only on $x$, whose numeric value may change from line to line unless otherwise specified. $a\lesssim_{x} b$ and $a\gtrsim_x b$ mean $a\leq C_x b$ and $a\geq C_x b$ respectively, and $a\asymp_x b$ means $a\lesssim_{x} b$ and $a\gtrsim_x b$ [$a\lesssim b$ means $a\leq Cb$ for some absolute constant $C$]. For two real numbers $a,b$, $a\vee b\equiv \max\{a,b\}$ and $a\wedge b\equiv\min\{a,b\}$. For two sequence of non-negative real numbers $\{a_n\},\{b_n\}$, $a_n\ll (\gg) b_n$ means $\lim_n a_n/b_n = 0 (\infty)$. We slightly abuse notation by defining $\log(x)\equiv \log(x\vee e)$ (and similarly for $\log \log (x)$).

%\subsection{Organization}

\section{Sampling designs}\label{section:sampling_designs}

\subsection{Setup}\label{section:setup}
Let $U_N\equiv \{1,\ldots,N\}$, and $\mathcal{S}_N\equiv \{\{s_1,\ldots,s_n\}: n \leq N, s_i \in U_N, s_i\neq s_j,\forall i\neq j\}$ be the collection of subsets of $U_N$. We adopt the super-population framework as in \cite{rubin2005two}: Let $\{(Y_i,Z_i) \in \mathcal{Y}\times \mathcal{Z}\}_{i=1}^N$ be i.i.d. super-population samples defined on a probability space $(\mathcal{X},\mathcal{A},\Prob_{(Y,Z)})$, where $Y^{(N)}\equiv (Y_1,\ldots,Y_N)$ is the vector of interest, and $Z^{(N)}\equiv (Z_1,\ldots,Z_N)$ is an auxiliary vector. A sampling design is a function $\mathfrak{p}: \mathcal{S}_N\times \mathcal{Z}^{\otimes N}\to [0,1]$ such that 
\begin{enumerate}
	\item for all $s \in \mathcal{S}_N$, $z^{(N)}\mapsto \mathfrak{p}(s,z^{(N)})$ is measurable,
	\item for all $z^{(N)} \in \mathcal{Z}^{\otimes N}$, $s\mapsto \mathfrak{p}(s, z^{(N)})$ is a probability measure.
\end{enumerate}

The probability space we work with that includes both the super-population and the design-space is the same product space $({\mathcal{S}}_N\times \mathcal{X},  \sigma({\mathcal{S}}_N)\times \mathcal{A}, \Prob)$ as constructed in \cite{boistard2017functional}. We include the construction here for convenience of the reader: the probability measure $\Prob$ is uniquely defined through its restriction on all rectangles: for any $(s,E) \in {\mathcal{S}}_N\times \mathcal{A}$ (note that $\mathcal{S}_N$ is a finite set),
\begin{align}\label{def:prob_meas}
\Prob\left(s\times E\right)\equiv \int_E \mathfrak{p}(s, z^{(N)}(\omega))\ \d{\Prob_{(Y,Z)}(\omega)}\equiv \int_E \Prob_d(s,\omega)\ \d{\Prob_{(Y,Z)}(\omega)}.
\end{align}
We also use $P$ to denote the marginal law of $Y$ for notational convenience.
%The above definition says that the sampling design $P$ can depend on the information induced by the auxiliary vector $Z^{(N)}$. 

Given $(Y^{(N)},Z^{(N)})$ and a sampling design $\mathfrak{p}$, let $\{\xi_i\}_{i=1}^N\subset [0,1]$ be random variables defined on $({\mathcal{S}}_N\times \mathcal{X},  \sigma({\mathcal{S}}_N)\times \mathcal{A}, \Prob)$ with  $\pi_i\equiv \pi_i(Z^{(N)})\equiv \E[\xi_i|Z^{(N)}]$. We further assume that $\{\xi_i\}_{i=1}^N$ are independent of $Y^{(N)}$ conditionally on $Z^{(N)}$. Typically we take $\xi_i\equiv \bm{1}_{i \in s}$, where $s\sim \mathfrak{p}$, to be the indicator of whether or not the $i$-th sample $Y_i$ is observed (and in this case $\pi_i(Z^{(N)})=\sum_{s \in \mathcal{S}_N: i \in s} \mathfrak{p}(s,Z^{(N)})$), but we do not require this structure apriori. The $\pi_i$'s are often referred to as the first-order inclusion probabilities, and $\pi_{ij}\equiv \pi_{ij}(Z^{(N)})\equiv \E[\xi_i\xi_j|Z^{(N)}]$ are the second-order inclusion probabilities.

We define the Horvitz-Thompson empirical measure and empirical process as follows: for $\{\pi_i\},\{\xi_i\},\{Y_i\}$ as above, 
\begin{align*}
\Prob_N^{\pi}(f)\equiv \frac{1}{N}\sum_{i=1}^N \frac{\xi_i}{\pi_i} f(Y_i), \quad f \in \mathcal{F},
\end{align*}
and the associated Horvitz-Thompson empirical process
\begin{align*}
\G_N^{\pi}(f)\equiv \sqrt{N}\big(\Prob_N^{\pi}-P\big)(f),\quad f \in \mathcal{F}.
%=\frac{1}{\sqrt{N}}\sum_{i=1}^N \left(\frac{\xi_i}{\pi_i}-1\right)f(Y_i). 
\end{align*}
The name of such an empirical process goes back to \cite{horvitz1952generalization}, in which $\Prob_N^\pi(Y)$ is used as an estimator for the population mean $P(Y)$. The usual empirical measure and empirical process (i.e. with $\xi_i/\pi_i\equiv 1$ for all $i=1,\ldots,N$) will be denoted by $\Prob_N,\G_N$ respectively.
%The goal of this section is to study various global and local behavior related to the Horvitz-Thompson empirical measure/empirical process.

\begin{assumption}\label{assumption:sampling_design}
Consider the following conditions on the sampling design $\mathfrak{p}$:

\, (A1) $\min_{1\leq i\leq N}\pi_i\geq \pi_0>0$.

\, (A2-LLN) $\frac{1}{N}\sum_{i=1}^N \big(\frac{\xi_i}{\pi_i}-1\big) =\mathfrak{o}_{\mathbf{P}}(1)$.

\, (A2-CLT) $\frac{1}{\sqrt{N}}\sum_{i=1}^N \big(\frac{\xi_i}{\pi_i}-1\big) =\mathcal{O}_{\mathbf{P}}(1)$.
%\begin{enumerate}
%	\item[(A1)] $\min_{1\leq i\leq N}\pi_i\geq \pi_0>0$.
%	\item[(A2\,\,]-LLN) $\frac{1}{N}\sum_{i=1}^N \big(\frac{\xi_i}{\pi_i}-1\big) =\mathfrak{o}_{\mathbf{P}}(1)$.
%	\item[(A2\,\,]-CLT) $\frac{1}{\sqrt{N}}\sum_{i=1}^N \big(\frac{\xi_i}{\pi_i}-1\big) =\mathcal{O}_{\mathbf{P}}(1)$.
%\end{enumerate}
\end{assumption}

(A1) is a common assumption in the literature. (A2-LLN) says that the weights $\{\xi_i/\pi_i\}$ satisfy a law of large numbers; while (A2-CLT) says that the weights $\{\xi_i/\pi_i\}$ have a $\sqrt{N}$ rate of convergence (so that a uniform central limit theorem for the more complicated Horvitz-Thompson empirical process $\G_N^{\pi}$ can be possible). As we will see below in the examples, a generic way of verifying these conditions is to obtain a good estimate on the correlations $\{\pi_{ij}-\pi_i\pi_j\}_{i\neq j}$. Conditions on (even higher order) correlations are very common in the literature, cf. \cite{boistard2012approximation,boistard2017functional,breidt2000local,cardot2010properties}.

\subsection{Examples of sampling designs}

\begin{example}[Sampling without replacement]\label{ex:sampling_without_replacement}
	A simple random sampling without replacement (SWOR) design $\mathfrak{p}$ is such that for all $z^{(N)} \in \mathcal{Z}^{\otimes N}$, $\mathfrak{p}(\cdot,z^{(N)})$ is the sampling without replacement design with cardinality $n(z^{(N)})$. In this case, $(\xi_1,\ldots,\xi_N)$ is a random permutation of $(1,\ldots,1,0,\ldots,0)$ that contains $1$ in the first $n(z^{(N)})$ components and $0$ otherwise. Then
	\begin{align*}
	\pi_i(z^{(N)}) = \E[\xi_i|z^{(N)}]= \frac{n(z^{(N)})}{N}.
	\end{align*}
	Condition (A1) holds if $n(z^{(N)})/N\geq c$ for some constant $c>0$. Condition (A2) is trivially satisfied since $\sum_{i=1}^N \xi_i = n(z^{(N)})$ and hence 
	\begin{align*}
	\sum_{i=1}^N\left(\frac{\xi_i}{\pi_i}-1\right)=\bigg( \frac{1}{n(z^{(N)})/N}\cdot \sum_{i=1}^N \xi_i\bigg) - N = 0.
	\end{align*}
\end{example}

\begin{example}[Bernoulli sampling]
	A Bernoulli sampling design $\mathfrak{p}$ is such that for all $z^{(N)} \in \mathcal{Z}^{\otimes N}$ and $s \in \mathcal{S}_N$, 
	\begin{align*}
	\mathfrak{p}(s,z^{(N)})= \prod_{ i \in s} \pi_i(z^{(N)}) \prod_{i \notin s} (1-\pi_i(z^{(N)})).
	\end{align*}
	In other words, conditionally on auxiliary random variables $Z^{(N)}$, the $\xi_i$'s are independent Bernoulli random variables with success probability $\pi_i(Z^{(N)})$. Note that we allow $\{\pi_i(Z^{(N)})\}$ to be unequal. Condition (A1) holds if $\pi_i(Z^{(N)})\geq c$ for some constant $c>0$. Since
	\begin{align*}
	\E \bigg(\frac{1}{\sqrt{N}}\sum_{i=1}^N \bigg(\frac{\xi_i}{\pi_i}-1\bigg) \bigg)^2
%	& = \E_{Y^{(N)},Z^{(N)}} \bigg[\E_{\xi^{(N)}| (Y^{(N)},Z^{(N)})}\bigg(\frac{1}{\sqrt{N}}\sum_{i=1}^N \left(\frac{\xi_i}{\pi_i}-1\right) \bigg)^2\bigg]\\
	& = \E_{ (Y^{(N)},Z^{(N)})} \bigg[\E_{\xi^{(N)} }\frac{1}{N}\sum_{i=1}^N \left(\frac{\xi_i}{\pi_i}-1\right)^2\bigg]=\mathcal{O}(1),
%	&\leq \pi_0^{-2} \E_{Y^{(N)},Z^{(N)}} \bigg[\frac{1}{N}\sum_{i=1}^N \pi_i(Z^{(N)})(1-\pi_i(Z^{(N)}))\bigg]\leq \pi_0^{-2},
	\end{align*}
	condition (A2) is satisfied.
\end{example}

\begin{example}[Rejective sampling and high entropy sampling]
A rejective sampling design $\mathfrak{r}$ maximizes the entropy functional $\mathfrak{p}\mapsto \sum_{s \in \mathcal{S}_N} \mathfrak{p}(s) \log (\mathfrak{p}(s))$ over all sampling designs of fixed size $n$ with the constraint that the first-order inclusion probabilities equal $(\pi_1,\ldots,\pi_N)$ (cf. \cite{hajek1981sampling}). $\mathfrak{r}$ can also be realized as a conditional Bernoulli sampling design with appropriate success probabilities $(p_1,\ldots,p_N)$: for all $z^{(N)} \in \mathcal{Z}^{\otimes N}$ and $s \in \mathcal{S}_N$, 
	\begin{align*}
	\mathfrak{r}(s,z^{(N)})\propto \prod_{ i \in s} p_i(z^{(N)}) \prod_{i \notin s} (1-p_i(z^{(N)}))\bm{1}_{\abs{s}=n}.
	\end{align*}
	where $\sum_{i=1}^N p_i(z^{(N)})=n$. The relationship between $p_i$ and $\pi_i$ is given in, e.g. the statement and proof of Theorem 5.1 of \cite{hajek1964asymptotic}.
	
	Condition (A1) holds if $\pi_i(Z^{(N)})\geq c$ for some constant $c>0$. Let $d_N\equiv \sum_{i=1}^N \pi_i(z^{(N)})\big(1-\pi_i(z^{(N)})\big)$, and suppose that there exists some constant $K>0$ such that for $N$ large enough
	\begin{align}\label{cond:rej_sampling}
	\frac{N}{d_N}\leq K.
	\end{align}
	Then we have
	\begin{align*}
	&\E \bigg(\frac{1}{\sqrt{N}}\sum_{i=1}^N \left(\frac{\xi_i}{\pi_i}-1\right) \bigg)^2\\
	%& = \E_{Y^{(N)},Z^{(N)}} \left[\E_{\xi^{(N)}| (Y^{(N)},Z^{(N)})}\left(\frac{1}{N}\sum_{i=1}^N \left(\frac{\xi_i}{\pi_i}-1\right) \right)^2\right]\\
	& = \E_{Y^{(N)},Z^{(N)}} \bigg[\E_{\xi^{(N)}}\frac{1}{N}\bigg( \sum_{i=1}^N \left(\frac{\xi_i}{\pi_i}-1\right)^2 +\sum_{i\neq j} \left(\frac{\xi_i}{\pi_i}-1\right)\left(\frac{\xi_j}{\pi_j}-1\right) \bigg)\bigg]\\
	&\lesssim 1+ \E_{Y^{(N)},Z^{(N)}}\bigg[N^{-1}\sum_{i\neq j}\abs{\pi_{ij}-\pi_i\pi_j} \bigg]=\mathcal{O}(1),
	\end{align*}
	where in the last inequality we used an old result due to Haj\'ek (cf. Theorem 5.2 of \cite{hajek1964asymptotic}). Hence condition (A2) is satisfied under (\ref{cond:rej_sampling}).

	Assuming (for simplicity) now $0<\inf_i \pi_i\leq \sup_i \pi_i<1$. Then Theorems 1 and 2 in \cite{berger1998rate} showed that high entropy designs satisfy a central limit theorem. More precisely, any sampling design $\mathfrak{p}$ with first-order inclusion probabilities $(\pi_1,\ldots,\pi_N)$ and the property that $D_{\mathrm{KL}}(\mathfrak{p}||\mathfrak{r})=\sum_{s \in \mathcal{S}_N} \mathfrak{p}(s) \log\frac{\mathfrak{p}(s)}{\mathfrak{r}(s)}\to 0$ satisfies a CLT. An alternative argument can be found in the discussions after Proposition \ref{prop:cov_G_pi} below. In particular, all such high entropy designs satisfy conditions (A1)-(A2-CLT) under $0<\inf_i \pi_i\leq \sup_i \pi_i<1$. The examples in this regard examined in \cite{berger1998rate} include Rao-Sampford sampling and successive sampling (under some scaling conditions).
\end{example}

\begin{example}[Stratified sampling]
	Suppose that ${U}_N$ is partitioned into $\{{U}_{N_1},\ldots,{U}_{N_k}\}$ according to the auxiliary variables $Z^{(N)}$ (we omit such dependence for simplicity). In other words, $\cup_{\ell=1}^k {U}_{N_\ell} = {U}_N$, ${U}_{N_\ell}\cap {U}_{N_{\ell'}}=\emptyset$ for $\ell \neq \ell'$ and $\abs{{U}_{N_\ell}}=N_\ell$ with $\sum_{\ell=1}^k N_\ell=N$. Let $n_1,\ldots,n_k$ be such that $\sum_{\ell=1}^k n_\ell = n$. Within each stratum ${U}_{N_\ell}$, we draw $n_\ell\leq N_\ell$ samples $s_\ell$ without replacement. The overall sample is $s=\cup_{\ell=1}^k s_\ell$. Similar to the calculations in Example \ref{ex:sampling_without_replacement}, since $\sum_{i \in s_\ell} \xi_i= n_\ell$, we have
	\begin{align*}
	\sum_{i=1}^N \left(\frac{\xi_i}{\pi_i}-1\right) = \sum_{\ell =1}^k \bigg(\frac{1}{n_\ell/N_\ell}\sum_{i \in s_\ell} \xi_i\bigg)-N =\bigg(\sum_{\ell=1}^k N_\ell\bigg)-N=0.
	\end{align*}
	Hence (A2) is satisfied. (A1) holds if $n_\ell/N_\ell\geq c$ for some constant $c>0$. 
%	It is also easy to see from the above calculation that conditions (A1)-(A2) are not specific to the particular two-phase sampling design (i.e. multiple phases can be allowed). 
%	We will not further go into details in this direction by increasing notational complexities.
\end{example}

\begin{example}[Stratified sampling with overlap]
Recently \cite{saegusa2018large} studied an interesting extension of the stratified sampling design as follows: suppose that $\{{U}_{N_1},\ldots,{U}_{N_k}\}\subset {U}_N$ are $k$ potentially overlapping `data sources' determined by the auxiliary variables $Z^{(N)}$, where $k$ is a fixed integer. Let $N_\ell \equiv \abs{{U}_{N_\ell}}$. For each source ${U}_{N_\ell}$, we draw $n_\ell\leq N_\ell$ samples $s_\ell$ without replacement. The overall sample is $s=\cup_{\ell=1}^k s_\ell$, which may include duplicate samples due to the overlapping nature of the data sources. This sampling scheme is also known as multiple-frame surveys, cf. \cite{hartley1962multiple,hartley1974multiple,lohr2006estimation}.

Let $\bar{\pi}_{i}^{(\ell)}\equiv n_\ell/N_\ell$ if $i \in {U}_{N_\ell}$ be the sampling probability of unit $i$ in the data source ${U}_{N_\ell}$, and let $\bar{\xi}_i^{(\ell)}$ be the indicator of whether or not unit $i$ is sampled in ${U}_{N_\ell}$. Following \cite{saegusa2018large}, we consider the following variant of the Horvitz-Thompson empirical measure (or \emph{Hartley empirical measure} as it is named in \cite{saegusa2018large}):
\begin{align*}
\Prob_N^H (f) \equiv \frac{1}{N}\sum_{i=1}^N \sum_{\ell =1}^{k} \frac{\bar{\xi}_{i}^{(\ell)}\rho_i^{(\ell)}  }{\bar{\pi}_i^{(\ell)}} \bm{1}_{i \in {U}_{N_\ell}} f(Y_i),
\end{align*}
and the associated (Hartley) empirical process
\begin{align*}
\G_N^H(f)\equiv \sqrt{N}\big(\Prob_N^H-P\big)(f).
\end{align*}
Here the weights $\{\rho_i^{(\ell)}\equiv \rho_i^{(\ell)}(z^{(N)}) \in [0,1] \}$ are such that $\sum_{\ell=1}^{k} \rho_i^{(\ell)}(z^{(N)})=1$ and that $\rho_{i}^{(\ell)}=0$ if $i\notin {U}_{N_\ell}$. Now letting
\begin{align}\label{def:hartley_xi}
\pi_i \equiv \prod_{\ell=1}^{k} \bar{\pi}_i^{(\ell)},\quad \xi_i\equiv \sum_{\ell=1}^{k} \bigg(\bm{1}_{i \in {U}_{N_\ell}}\bar{\xi}_i^{(\ell)} \rho_i^{(\ell)} \prod_{\ell' \neq \ell} \bar{\pi}_i^{(\ell')}\bigg) \in [0,1],
\end{align}
we see that the Hartley empirical measure $\Prob_N^H$ and the associated empirical process $\G_N^H$ reduces to the Horvitz-Thompson empirical measure and empirical process with $\{\pi_i,\xi_i\}$ specified in (\ref{def:hartley_xi}).

Condition (A1) holds if $n_\ell/N_\ell\geq c$ for some constant $c>0$ (by noting that $k$ is a fixed constant that does not depend on $Z^{(N)}$). Now we verify (A2). Note that
\begin{align*}
\frac{1}{\sqrt{N}}\sum_{i=1}^N \left(\frac{\xi_i}{\pi_i}-1\right) &= \frac{1}{\sqrt{N}}\bigg[\sum_{i=1}^N \sum_{\ell =1}^{k} \frac{\bar{\xi}_{i}^{(\ell)}\rho_i^{(\ell)}  }{\bar{\pi}_i^{(\ell)}} \bm{1}_{i \in {U}_{N_\ell}} -N\bigg]\\
& = \sum_{\ell =1}^{k} \frac{1}{\sqrt{N}}\sum_{i=1}^N \bigg(\frac{\bar{\xi}_{i}^{(\ell)} }{\bar{\pi}_i^{(\ell)}}-1\bigg)\rho_i^{(\ell)}  \bm{1}_{i \in {U}_{N_\ell}}=\mathcal{O}_{\mathbf{P}}(1),
\end{align*}
where the last line follows by computing the second moment:
\begin{align*}
&\E \bigg[\frac{1}{\sqrt{N}}\sum_{i=1}^N \bigg(\frac{\bar{\xi}_{i}^{(\ell)} }{\bar{\pi}_i^{(\ell)}}-1\bigg)\rho_i^{(\ell)}  \bm{1}_{i \in {U}_{N_\ell}}\bigg]^2\\
&\lesssim 1+ \frac{1}{N}\sum_{i\neq j \in {U}_{N_\ell}}\E_{(Y^{(N)},Z^{(N)})} \bigg[ \biggabs{\E_{\xi^{(N)} }\bigg(\frac{\bar{\xi}_{i}^{(\ell)} }{\bar{\pi}_i^{(\ell)}}-1\bigg)\bigg(\frac{\bar{\xi}_{j}^{(\ell)} }{\bar{\pi}_j^{(\ell)}}-1\bigg)} \bigg]= \mathcal{O}(1).
\end{align*}
This verifies (A2-CLT). 

From the above derivation it is easy to see that (A1)-(A2-CLT) hold with the sampling without replacement design replaced by Bernoulli sampling and rejective sampling designs.

We also note that different choices of the weights $\{\rho_i^{(\ell)}\equiv \rho_i^{(\ell)}(z^{(N)}) \in [0,1] \}$ lead to different asymptotic variances. Since this issue is not the main concern of this paper, we refer the readers to \cite{saegusa2018large} for the optimal choice of weights in the context of Bernoulli sampling and sampling without replacement designs.
\end{example}

\section{Theory}\label{section:theory}

In this section, we will be mainly interested in the global and local behavior of the Horvitz-Thompson empirical process. In particular, we prove a Glivenko-Cantelli theorem and a Donsker theorem that provide global information concerning the Horvitz-Thompson empirical process in the limit. As will be seen, our formulation requires almost minimal conditions. We further study local behavior of the Horvitz-Thompson empirical process by characterizing its local asymptotic modulus and several ratio limit theorems. Understanding the local behavior of the Horvitz-Thompson empirical process plays a key role in applications to statistical problems as will be demonstrated in Section \ref{section:applications}. Corresponding results for the calibrated version of the Horvitz-Thompson empirical process are also included. We also discuss uniform limit theorems for some variants of the Horvitz-Thompson empirical process and their conditional versions thereof.

\subsection{Global and local limit theorems}

First we study the Glivenko-Cantelli theorem. We say that $\mathcal{F}$ is $P$-Glivenko-Cantelli if and only if $\sup_{f \in \mathcal{F}}\abs{(\Prob_N-P)(f)}=\mathfrak{o}_{\mathbf{P}}(1)$.

\begin{theorem}\label{thm:ULLN_HT}(Glivenko-Cantelli Theorem)
	Suppose that (A1) and (A2-LLN) hold. If $\mathcal{F}$ is $P$-Glivenko-Cantelli, then 
	\begin{align*}
	\sup_{f \in \mathcal{F}}\abs{(\Prob_N^\pi-P)(f)}=\mathfrak{o}_{\mathbf{P}}(1).
	\end{align*}
\end{theorem}

Recall the notion of weak convergence in the Hoffmann-J\o rgensen sense: Let $\{X(f)\}_{f \in \mathcal{F}}$ be a bounded process whose finite-dimensional laws correspond to the finite dimensional projections of a tight Borel law on $\ell^\infty(\mathcal{F})$. Let $\{X_N(f)\}_{f \in \mathcal{F}}$ be bounded processes. We say that $X_N \rightsquigarrow X$ in $\ell^\infty(\mathcal{F})$ if and only if $
\E^\ast H(X_N)\to \E H(\tilde{X})$ for all $H \in C_b(\ell^\infty(\mathcal{F}))$, where $C_b(\ell^\infty(\mathcal{F}))$ denotes all bounded continuous functions on $\ell^\infty(\mathcal{F})$, and $\tilde{X}$ is a measurable version of $X$ with separable range (so $H(\tilde{X})$ is measurable). Equivalently, $d_{\mathrm{BL}}(X_N,\tilde{X})\to 0$, where $d_{\mathrm{BL}}$ is the dual bounded Lipschitz metric (cf. pp 246 of \cite{gine2015mathematical}). It is also well-known that $X_N \rightsquigarrow X$ in $\ell^\infty(\mathcal{F})$ if and only if $X_N$ converges to $X$ finite-dimensionally, and there exists a pseudo-metric $d$ on $\mathcal{F}$ such that for any $\delta_N\to 0$,
\begin{align*}
\sup_{d(f,g)\leq \delta_N} \abs{X_N(f)-X_N(g)}=\mathfrak{o}_{\mathbf{P}}(1).
\end{align*}
We refer the readers to \cite{gine2015mathematical,van1996weak} for more details. We say that $\mathcal{F}$ is $P$-Donsker if and only if $\G_N\rightsquigarrow \G$ in $\ell^\infty(\mathcal{F})$.

\begin{theorem}\label{thm:weak_convergence_HT}(Donsker Theorem)
Suppose that (A1) and (A2-CLT) hold. Further assume that
\begin{enumerate}
\item[(D1)] $\G_N^{\pi}$ converges finite-dimensionally to a tight Gaussian process $\G^{\pi}$.
\item[(D2)] $\mathcal{F}$ is $P$-Donsker.
\end{enumerate}
Then 
\begin{align*}
\G_N^{\pi}\rightsquigarrow \G^{\pi}\textrm{ in }\ell^\infty(\mathcal{F}).
\end{align*}
\end{theorem}

Apparently, the finite-dimensional convergence condition (D1) above is necessary for a uniform central limit theorem in $\ell^\infty(\mathcal{F})$. (D2) is also minimal. One intriguing feature of Theorem \ref{thm:weak_convergence_HT} is that a uniform central limit theorem follows essentially automatically as long as the \emph{finite-dimensional convergence property of the Horvitz-Thompson empirical process is verified}. 
%Similar formulations can be found in Theorem 2.11.1, 2.11.9, 2.11.11 in \cite{van1996weak} in the context of Donsker theorems for general processes. 
A similar phenomenon was also observed in \cite{shorack1973convergence} in a univariate non-i.i.d. case.

Although being necessary, establishing a finite-dimensional CLT for $\G_N^\pi$ and identifying the covariance structure of $\G^\pi$ can be a non-trivial problem for general sampling designs; see e.g. \cite{berger1998ratevariance,berger1998rate,chauvet2015coupling,fuller2011sampling,hajek1964asymptotic,rosen1965limit,rosen1967central,rosen1972asymptotic,rosen1974asymptotic,vivsek1979asymptotic}. Below we exploit one possible strategy, inspired by \cite{boistard2017functional}, for identifying the covariance structure of $\G^\pi$.
\begin{proposition}\label{prop:cov_G_pi}
	Suppose (A1) and the following conditions hold.
	\begin{enumerate}
		\item[(F1)] For any i.i.d. bounded random variables $\{V_i\}$ defined on $(\mathcal{X},\mathcal{A},\Prob_{(Y,Z)})$, 
		\begin{align*}
		\frac{1}{S_N}\bigg(\frac{1}{N}\sum_{i=1}^N \frac{\xi_i}{\pi_i} V_i - \frac{1}{N}\sum_{i=1}^N V_i\bigg) \rightsquigarrow \mathcal{N}(0,1)
		\end{align*}
		holds under $\Prob_d(\cdot,\omega)$ (notation defined in (\ref{def:prob_meas})) for $\Prob_{(Y,Z)}$-a.s. $\omega \in \mathcal{X}$. Here $S_N$ is the design-based variance given by
		\begin{align*}
		S_N^2\equiv \frac{1}{N^2}\sum_{1\leq i,j\leq N} \frac{\pi_{ij}-\pi_i\pi_j}{\pi_i\pi_j} V_iV_j.
		\end{align*}
		\item[(F2)] The (essentially) first-order inclusion probabilities satisfy
		\begin{align*}
		 \frac{1}{N}\sum_{i=1}^N \frac{\pi_{ii}-\pi_i^2}{\pi_i^2}\to_{\Prob_{(Y,Z)}}\mu_{\pi1}.
		\end{align*}
		\item[(F3)] The second-order inclusion probabilities satisfy 
		\begin{align*}
		\sup_{N \in \N} \sup_{1\leq i\neq j\leq N} N\abs{\pi_{ij}-\pi_i\pi_j}\leq K,\qquad  \frac{1}{N}\sum_{i\neq j}\frac{\pi_{ij}-\pi_i\pi_j}{\pi_i\pi_j}\to_{\Prob_{(Y,Z)}}\mu_{\pi2},
		\end{align*}
		where $K>0$ is an absolute constant.
	\end{enumerate}
	If $\mathcal{F}$ is uniformly bounded, then $\G_N^\pi$ converges finite-dimensionally to a tight Gaussian process $\G^\pi$ whose covariance structure is given by the following: for any $f,g \in \mathcal{F}$,
	\begin{align*}
	\mathrm{Cov} \big(\G^\pi(f), \G^\pi(g)\big)&= (1+\mu_{\pi 1}) P (fg) - (1-\mu_{\pi 2}) (Pf)(Pg)\\
	& = P(fg)-(Pf)(Pg)+ \mu_{\pi1}P(fg)+\mu_{\pi2}(Pf)(Pg).
	\end{align*}
\end{proposition}
The above covariance formula can be inferred from the decomposition 
\begin{align*}
\G_N^\pi = \sqrt{N}(\Prob_N^\pi - P) = \sqrt{N}(\Prob_N-P)+\sqrt{N}(\Prob_N^\pi-\Prob_N),
\end{align*}
where the covariance structure of the second term $\sqrt{N}(\Prob_N^\pi-\Prob_N)$ can be deduced from conditions (F1)-(F3). These conditions are also used in \cite{boistard2017functional}:  (F1) corresponds to (HT1) in \cite{boistard2017functional}, (F2) corresponds to condition (i) in Proposition 3.1 in \cite{boistard2017functional}, and (F3) corresponds to (C2) and condition (ii) in Proposition 3.1 in \cite{boistard2017functional}. Combined with Proposition \ref{prop:cov_G_pi}, we see that Theorem \ref{thm:weak_convergence_HT} extends Proposition 3.2 of \cite{boistard2017functional} in at least the following directions: (i) we work with a general bounded $P$-Donsker class $\mathcal{F}$ instead of one particular class $\{\bm{1}_{(-\infty,t]}: t \in \R\}$, and (ii) we weaken conditions for the sampling designs, i.e. (C3)-(C4) in \cite{boistard2017functional} are no longer required. We should, however, remind readers that Proposition \ref{prop:cov_G_pi} is not exhaustive for identifying the covariance structure of $\G^\pi$, and therefore it is possible that the current conditions in Proposition \ref{prop:cov_G_pi} can be further weakened via other approaches.

The conditions in Proposition \ref{prop:cov_G_pi} are verified in \cite{boistard2017functional} under a slightly different setting, but for the convenience of the reader, we provide some details for various sampling designs (see Table \ref{table:mu_sampling_design} for a summary):

\begin{itemize}
\item For sampling without replacement, $\pi_{ii}=\pi_i=n/N$ and $\pi_{ij}=n(n-1)/N(N-1)$ for $i\neq j$. If $n/N\to \lambda \in (0,1)$, (F1) can be verified using Haj\'ek's rank central limit theorem (cf. \cite{hajek1961some}, or Proposition A.5.3 of \cite{van1996weak}), and (F2)-(F3) are satisfied with $\mu_{\pi 1}=\lambda^{-1}-1$ and $\mu_{\pi2}=1-\lambda^{-1}$. The cases for stratified sampling with/without overlaps can be considered analogously.
\item For Bernoulli sampling, $\pi_{ii}=\pi_i$ and $\pi_{ij}=\pi_i\pi_j$ for $i\neq j$. If $\{\pi_i\}_{i=1}^N \subset [\epsilon,1-\epsilon] (\epsilon>0)$, (F1) can be verified using the Lindeberg-Feller central limit theorem, and (F2)-(F3) are satisfied with $\mu_{\pi 1}= \lim_{N} N^{-1}\sum_{i=1}^N(\pi_i^{-1}-1)$ and $\mu_{\pi 2}=0$. 
\item For rejective sampling with first-order inclusion probabilities $\{\pi_i\}_{i=1}^N \subset [\epsilon,1-\epsilon] (\epsilon>0)$, let $d_N=\sum_{i=1}^N \pi_i(1-\pi_i)$. (F1) can be verified by Theorem 1 of \cite{berger1998rate}. Using Theorem 1 of \cite{boistard2012approximation}, (F2)-(F3) are satisfied with $\mu_{\pi 1} = \lim_N N^{-1}\sum_{i=1}^N (\pi_i^{-1}-1)$ and 
\begin{align*}
\qquad\qquad\mu_{\pi2}=\lim_{N} \bigg[-\frac{1}{N}\sum_{i\neq j} \frac{(1-\pi_i)(1-\pi_j)}{d_N}+\mathcal{O}(Nd_N^{-2})\bigg]=- d^{-1}(1-\lambda)^2,
\end{align*}
provided $n/N\to \lambda \in (0,1)$ and $d_N/N\to d$. The covariance structure of $\G^\pi$ with high entropy sampling designs is the same as the rejective sampling design, which can be verified using the same arguments in page 1754-1755 of \cite{boistard2017functional}.
\end{itemize}

\begin{table}
	\begin{center}
	\begin{tabular}{|c||c|c|c|}
		\hline 
		& SWOR & Bernoulli  &Rejective\\
		\hline\hline
		$\mu_{\pi1}$ & $\lambda^{-1}-1$ & $A-1$&$A-1$ \\
		\hline
		$\mu_{\pi2}$ & $1-\lambda^{-1}$ & $0$& $-d^{-1}(1-\lambda)^2$\\
		\hline
	\end{tabular}
\end{center}
\caption{Values of $\mu_{\pi1}, \mu_{\pi2}$ for different sampling designs. Here $\lambda = \lim_N n/N, A = \lim_N N^{-1} \sum_{i=1}^N \pi_i^{-1}, d=\lim_N N^{-1}\sum_{i=1}^N \pi_i(1-\pi_i)$. }
\label{table:mu_sampling_design}
\vspace{-2em}
\end{table}

Hence, under the assumptions of Proposition \ref{prop:cov_G_pi}, the covariance formula for $\G^\pi$ can be written more explicitly: for any $f,g \in \mathcal{F}$,
\begin{align*}
&\mathrm{Cov} \big(\G^\pi(f), \G^\pi(g)\big)\\
& = 
\begin{cases}
\lambda^{-1}\big(P(fg)-(Pf)(Pg)\big) & \textrm{under SWOR}\\
A \cdot P(fg)-(Pf)(Pg) & \textrm{under Bernoulli}\\
A\cdot P(fg)-\big[1+d^{-1}(1-\lambda)^2\big](Pf)(Pg) & \textrm{under Rejective}\\
\end{cases}
\end{align*}
Here $\lambda = \lim_N n/N, A = \lim_N N^{-1}\sum_{i=1}^N \pi_i^{-1}, d = \lim_N N^{-1}\sum_{i=1}^N \pi_i(1-\pi_i)$.

%\begin{remark}
%In the typical case $\sum_{i=1}^N \pi_i =1$, since the harmonic mean is no larger than the arithmetic mean, we always have 
%\begin{align*}
%A^{-1} = \lim_N \bigg(N^{-1}\sum_{i=1}^N \pi_i^{-1}\bigg)^{-1}\leq \lim_N \bigg(N^{-1}\sum_{i=1}^N \pi_i \bigg) = \lim_N n/N = \lambda.
%\end{align*}
%\end{remark}

Our next goal is to study the local behavior of the Horvitz-Thompson empirical process. Although being of crucial importance in applications to semi-/non-parametric statistics, to the best knowledge of the authors, this issue has not been addressed in the literature. 

We first study \emph{local asymptotic modulus} of the Horvitz-Thompson empirical process, which has been considered historically for VC-type classes of sets and function classes in \cite{alexander1987rates,gine2006concentration,gine2003ratio} in the context of usual empirical processes. As will be clear below, one of the strengths of the formulation of our theorems is that finite-dimensional convergence of $\G_N^\pi$ is not required for studying the local behavior of $\G_N^\pi$---we only require that the weights have a $\sqrt{N}$ convergence rate as in (A2-CLT).

Before formally stating the results on the local behavior of the Horvitz-Thompson empirical process, we need some definitions.

\begin{definition}\label{def:local_asymp_moduli}
	A \emph{local asymptotic modulus} of the Horvitz-Thompson empirical process indexed by a class of functions $\mathcal{F}$ is an increasing function $\phi(\cdot)$ for which there exist some $r_N\ll \delta_N\leq 1/2$, both non-increasing with $N\mapsto \sqrt{N}\delta_N$ non-decreasing, such that
	\begin{align}\label{def:local_asymp_moduli_1}
    \sup_{f \in \mathcal{F}: r_N^2<Pf^2\leq \delta_N^2} \frac{\abs{\G_N^\pi(f)}}{\phi(\sigma_P f)}=\mathcal{O}_{\mathbf{P}}(1).
	\end{align}
	Here $\sigma_P^2(f) = \mathrm{Var}_P(f)$.
\end{definition}

\begin{definition}
We say that $\mathcal{F}$ satisfies an entropy condition with exponent $ \alpha \in (0,2)$ if either
\begin{align*}
\sup_Q \log \mathcal{N}(\epsilon \| F \|_{L_2 (Q)}, \mathcal{F} , L_2 (Q) ) \lesssim \epsilon^{-\alpha},
\end{align*}
where the supremum is over all finitely discrete measures $Q$ on $(\mathcal{X},\mathcal{A})$; or
\begin{align*}
\log \mathcal{N}_{[\, ]} (\epsilon , \mathcal{F}, L_2 (P) ) \lesssim \epsilon^{-\alpha}.
\end{align*}
\end{definition}

The entropy condition is well-understood in the literature; we only refer the readers to \cite{gine2015mathematical,van2000empirical,van1996weak} for various examples in this regard.

\begin{theorem}\label{thm:local_asymp_moduli}
Suppose that (A1) and (A2-CLT) hold and $\mathcal{F}$ is a uniformly bounded class satisfying an entropy condition with exponent $\alpha \in (0,2)$. Then $\omega_\alpha(t)=t^{1-\frac{\alpha}{2}}$ is a local asymptotic modulus for the Horvitz-Thompson empirical process indexed by $\mathcal{F}$, i.e. (\ref{def:local_asymp_moduli_1}) holds with $\phi=\omega_\alpha$.
\end{theorem}

The local asymptotic modulus is a key step in understanding the behavior of the Horvitz-Thompson empirical process at a local level. This will be useful in applications in the next section. The local asymptotic modulus $\omega_\alpha$ cannot be improved in general; this can be shown for the usual empirical process indexed by \emph{$\alpha$-full class} (which essentially requires a lower bound for the entropy number in a more local sense, cf. \cite{gine2006concentration}).

One may also invert the above viewpoint by fixing one particular weight function $\phi$ and asking for the rate of convergence of the corresponding weighted Horvitz-Thompson empirical process. Below are two particular choices: the first one (\ref{ineq:ratio_HT_rate_1}) uses $\phi(x)=x$, and the second one (\ref{ineq:ratio_HT_rate_2}) uses (essentially) $\phi(x)=x^2$.

\begin{theorem}\label{thm:ratio_HT_rate}
	Suppose that (A1) and (A2-CLT) hold and $\mathcal{F}$ is a uniformly bounded class satisfying an entropy condition with exponent $\alpha \in (0,2)$. Let $r_N\gtrsim N^{-1/(\alpha+2)}$. Then
	\begin{align}\label{ineq:ratio_HT_rate_1}
	 N^{1/(\alpha+2)} \sup_{f \in \mathcal{F}: \sigma_P f\geq r_N} \frac{\abs{(\Prob_N^\pi -P)(f)}}{\sigma_P f}=\mathcal{O}_{\mathbf{P}}(1).
	\end{align}
	If furthermore $\mathcal{F}$ takes value in $[0,1]$, then for any $L_N \to \infty$,
	\begin{align}\label{ineq:ratio_HT_rate_2}
	\sup_{f \in \mathcal{F}: Pf\geq L_N \cdot r_N} \biggabs{\frac{\Prob_N^\pi f}{Pf}-1}= \mathfrak{o}_{\mathbf{P}}(1).
	\end{align}
\end{theorem}

Results analogous to (\ref{ineq:ratio_HT_rate_1})-(\ref{ineq:ratio_HT_rate_2}) have been derived in the case of i.i.d. sampling in \cite{mason1983strong,shorack1982limit,stute1982oscillation,stute1984oscillation,wellner1978limit} for uniform empirical processes on (subsets of) $\R$ (or $\R^d$), and are further investigated in \cite{alexander1987rates} for VC classes of sets, and extended by \cite{gine2006concentration,gine2003ratio} who studied more general VC-subgraph classes. 

Note that (\ref{ineq:ratio_HT_rate_2}) can be viewed as a uniform law of large numbers for the weighted Horvitz-Thompson empirical process. We can also establish a central limit theorem for the weighted Horvitz-Thompson empirical process, analogous to the development in \cite{alexander1985rates,alexander1987central,alexander1987rates,gine2006concentration} for the usual empirical process. 

\begin{theorem}\label{thm:weighted_CLT}
	Suppose that (A1) and (A2-CLT) hold, and that $\mathcal{F}$ is a uniformly bounded class satisfying an entropy condition with exponent $\alpha \in (0,2)$. Let $\phi: \R_{\geq 0}\to \R_{\geq 0}$ be such that $\phi(0)=0$ and that 
	\begin{align}\label{cond:weight_CLT}
	\frac{\phi(t)}{t^{1-\frac{\alpha}{2}}(\log \log (1/t))^{1/2}}\to \infty
	\end{align}
	as $t \to 0$. If $r_N\gtrsim N^{-1/(\alpha+2)}$ and $\G_N^{\pi}$ converges finite-dimensionally to a tight Gaussian process $\G^{\pi}$, then
	\begin{align*}
	\frac{\G_N^{\pi}(f)}{\phi(\sigma_P f)}\bm{1}_{\sigma_P f>r_N}\rightsquigarrow \frac{\G^{\pi}(f)}{\phi(\sigma_P f)}\quad \textrm{ in }\ell^\infty(\mathcal{F}).
	\end{align*}
\end{theorem}

The weight function in the above theorem is required to be only slightly stronger than the local asymptotic modulus by an iterated logarithmic factor. This is very natural: the weight function cannot beat the local asymptotic modulus for a weighted CLT to hold, so the condition (\ref{cond:weight_CLT}) is optimal up to an iterated logarithmic factor.

\subsection{Calibration}

In practice, since the Horvitz-Thompson estimator may be severely inefficient, calibration of the weights is often used to improve efficiency \cite{deville1992calibration,lumley2011connections}. The main purpose of this section, instead of proposing new calibration methods or addressing efficiency issues, rests in demonstrating that our theoretical results are still valid for the Horvitz-Thompson empirical process with calibrated weights. 

To illustrate this, we consider one popular calibration method that aims at matching the population mean for the Horvitz-Thompson estimator \cite{deville1992calibration}. Let $\mathcal{Z}\subset\R^d$ be a compact set, and $G:\R \to \R_{\geq 0}$. Let $\hat{\alpha}_N \in \mathcal{A}_c$, where $\mathcal{A}_c$ is a compact set of $\R^d$, be defined via
\begin{align*}
\frac{1}{N}\sum_{i=1}^N \frac{\xi_i G(Z_i^\top \hat{\alpha}_N)}{\pi_i }Z_i = \frac{1}{N}\sum_{i=1}^N Z_i.
\end{align*}
Then the \emph{calibrated Horvitz-Thompson empirical measure} and \emph{calibrated Horvitz-Thompson empirical process} are defined by
\begin{align*}
\Prob_N^{\pi,c} (f) \equiv \frac{1}{N}\sum_{i=1}^N \frac{\xi_i G(Z_i^\top \hat{\alpha}_N)}{\pi_i} f(Y_i),\quad f \in \mathcal{F},
\end{align*}
and
\begin{align*}
\G_N^{\pi,c} (f)\equiv \sqrt{N} \big(\Prob_N^{\pi,c}-P\big)(f),\quad f \in \mathcal{F}
\end{align*}
respectively. 

Our next theorem asserts that as long as $\hat{\alpha}_N$ converges to the `truth' $0$ (which can be defined to be another value, but we use $0$ for notational convenience) sufficiently fast, the global and local theorems also hold for the calibrated Horvitz-Thompson empirical process.
\begin{theorem}\label{thm:calibrate_HT}
Suppose $G(0)=1$, $G'(0)>0$. Let $\mathcal{F}$ be a class of measurable functions with a measurable envelope $F$. 
\begin{enumerate}
\item[(1)] Let the assumptions in Theorem \ref{thm:ULLN_HT} hold with $PF<\infty$. If $\hat{\alpha}_N =\mathfrak{o}_{\mathbf{P}}(1)$, then the conclusion of Theorem \ref{thm:ULLN_HT} holds with $\Prob_N^\pi$ replaced by $\Prob_N^{\pi,c}$.
\item[(2)] Let the assumptions in Theorem \ref{thm:weak_convergence_HT} hold with $PF^2<\infty$ (but the finite-dimensional convergence condition is replaced by ${\G}_N^{\pi,c}$ converges finite-dimensionally to some tight Gaussian process ${\G}^{\pi,c}$). If $\sqrt{N}\hat{\alpha}_N=\mathcal{O}_{\mathbf{P}}(1)$, then
\begin{align*}
\G_N^{\pi,c}\rightsquigarrow {\G}^{\pi,c}\textrm{ in }\ell^\infty(\mathcal{F}).
\end{align*}
\item[(3)] If $\sqrt{N}\hat{\alpha}_N=\mathcal{O}_{\mathbf{P}}(1)$, then under the same conditions as in Theorems \ref{thm:local_asymp_moduli}, \ref{thm:ratio_HT_rate} and \ref{thm:weighted_CLT} (but the finite-dimensional convergence condition is replaced by ${\G}_N^{\pi,c}$ converges finite-dimensionally to some tight Gaussian process ${\G}^{\pi,c}$), the respective conclusions hold for the calibrated Horvitz-Thompson empirical process.
\end{enumerate}
\end{theorem}

The structural commonality in the above theorem is characterized by the $\sqrt{N}$-rate of the estimate $\hat{\alpha}_N$. Establishing a $\sqrt{N}$-rate for $\hat{\alpha}_N$ is not hard: in fact we can use Theorem 3.3.1 of \cite{van1996weak} for such a purpose by verifying the asymptotic equi-continuity of the Horvitz-Thompson empirical process. 

Below we exploit one possible strategy for this via the method of Proposition \ref{prop:cov_G_pi}. For simplicity of exposition, we assume that $\pi_i\equiv \pi_i(Z_i)$.

\begin{proposition}\label{prop:cov_structure_cal}
Assume the conditions of Proposition \ref{prop:cov_G_pi} and Theorem \ref{thm:calibrate_HT} hold. Further assume that $G$ is continuous with its derivative $G'$ locally continuous at $0$, and the map $\alpha \mapsto P[G(Z^\top \alpha-1) Z]$ has a unique zero at $0$, and $P(ZZ^\top) \in \R^{d\times d}$ is invertible. Then
\begin{align}\label{ineq:cov_structure_cal_0}
\sqrt{N}\hat{\alpha}_N=-(G'(0))^{-1} (P(ZZ^\top))^{-1} (\G_N^\pi-\G_N)Z+\mathfrak{o}_{\mathbf{P}}(1).
\end{align}
Furthermore, $\G_N^{\pi,c}$ converges finite-dimensionally to a tight Gaussian process $\G^{\pi,c}$ whose covariance structure is given by the following: for any $f,g \in \mathcal{F}$,
\begin{align*}
&\mathrm{Cov} \big( \G^{\pi,c}(f), \G^{\pi,c}(g)\big)\\
 & \quad\quad= P(fg)-(Pf)(Pg)+\mu_{\pi 1}P\big(\mathcal{T}(f)\mathcal{T}(g)\big)+\mu_{\pi2}(P\mathcal{T}(f))(P\mathcal{T}(g)).
\end{align*}
Here the operator $\mathcal{T}: \R^{\mathcal{Y}\times \mathcal{Z}}\to \R^{\mathcal{Y}\times \mathcal{Z}}$ is defined by 
\begin{align*}
\mathcal{T}(f)(y,z)=f(y)-P(f(Y)Z^\top) (P(ZZ^\top))^{-1} z.
\end{align*}
\end{proposition}

As we will see in the proofs, the asymptotic expansion for $\sqrt{N}\hat{\alpha}_N$ in (\ref{ineq:cov_structure_cal_0}) plays a crucial role in identifying the covariance structure of $\G^{\pi,c}$. Although here we only study one particular calibration method that matches the population mean, other calibration methods are also possible. Typically different calibration methods only differ in terms of the exact form of the corresponding operators $\mathcal{T}$; see e.g. \cite{saegusa2013weighted} for various calibration methods under the (two-phase) stratified sampling design.

\subsection{Other variants}

Our global limit theorems in Theorems \ref{thm:ULLN_HT} and \ref{thm:weak_convergence_HT} can be used for several other variants of the Horvitz-Thompson empirical processes studied in \cite{boistard2017functional}. We illustrate this by considering Donsker theorems for the variants as detailed below.

First consider $\sqrt{n}(\Prob_N^\pi - \Prob_N)$. We have the following:
\begin{corollary}\label{cor:P_N_center_Donsker}
	Suppose that (A1) and (A2-CLT) hold, and that $\mathcal{F}$ is $P$-Donsker. Further suppose that the conditions in Proposition \ref{prop:cov_G_pi} hold, and that $n/N\to \lambda \in (0,1)$. Then $\sqrt{n}(\Prob_N^\pi - \Prob_N)$ converges weakly in $\ell^\infty(\mathcal{F})$ to a Gaussian process $\bar{\G}^\pi$ whose covariance structure is given by the following: for any $f,g \in \mathcal{F}$,
	\begin{align*}
	&\mathrm{Cov}(\bar{\G}^\pi(f), \bar{\G}^\pi (g)) = \lambda \big(\mu_{\pi 1} P(fg)+\mu_{\pi 2}(Pf)(Pg)\big)\\
& = 
\begin{cases}
(1-\lambda)\big(P(fg)-(Pf)(Pg)\big) & \textrm{under SWOR}\\
\lambda(A-1) \cdot P(fg) & \textrm{under Bernoulli}\\
\lambda\big((A-1)\cdot P(fg)-d^{-1}(1-\lambda)^2(Pf)(Pg)\big) & \textrm{under Rejective}\\
\end{cases}
\end{align*}
Here $\lambda = \lim_N n/N, A = \lim_N N^{-1}\sum_{i=1}^N \pi_i^{-1}, d = \lim_N N^{-1}\sum_{i=1}^N \pi_i(1-\pi_i)$.
\end{corollary}

The covariance formula above is a direct consequence of the assumptions in Proposition \ref{prop:cov_G_pi}. Furthermore, the above corollary extends Theorem 3.1 of \cite{boistard2017functional} from the one-dimensional case $\mathcal{F}=\{\bm{1}_{(-\infty,t]}: t\in \R\}$ to a general setting. 
%Note that since the harmonic mean is less than the arithmetic mean,
%\begin{eqnarray*}
%A^{-1} = \lim_N \left ( N^{-1} \sum_{i=1}^N \pi_i^{-1} \right )^{-1}  \le \lim_N \left ( N^{-1} \sum_{i=1}^N \pi_i \right ) = \lim_N \frac{n}{N} = \lambda ,
%\end{eqnarray*}
%it follows that $A \ge 1/\lambda$ and hence $A-1 \ge \lambda^{-1} -1$ and then 
%\begin{eqnarray*}
%\lambda (A-1) \ge 1 - \lambda .
%\end{eqnarray*}

Next consider the H\'ajek empirical process. Let
\begin{align*}
\Prob_N^{\pi, H}(f) \equiv \frac{1}{\hat{N}}\sum_{i=1}^N \frac{\xi_i}{\pi_i} f(Y_i),\quad \hat{N} \equiv \sum_{i=1}^N \frac{\xi_i}{\pi_i}
\end{align*} 
be the H\'ajek empirical measure. We have the following:
\begin{corollary}\label{cor:Hajek_Donsker}
Suppose that (A1) and (A2-CLT) hold, and that $\mathcal{F}$ is $P$-Donsker. Further suppose that the conditions in Proposition \ref{prop:cov_G_pi} hold, and that $n/N\to \lambda \in (0,1)$. Then $\sqrt{n}\big(\Prob_N^{\pi, H}-\Prob_N\big)$ converges weakly to a Gaussian process $\bar{\G}^{\pi,H}$ whose covariance structure is given by the following: for any $f,g \in \mathcal{F}$,
\begin{align*}
&\mathrm{Cov}(\bar{\G}^{\pi,H}(f), \bar{\G}^{\pi,H} (g)) = \lambda \mu_{\pi 1} \big(P(fg)-(Pf)(Pg)\big)\\
& = 
\begin{cases}
(1-\lambda)\big(P(fg)-(Pf)(Pg)\big) & \textrm{under SWOR}\\
\lambda(A-1) \cdot \big(P(fg)-(Pf)(Pg)\big)& \textrm{under Bernoulli}\\
\lambda(A-1) \cdot \big(P(fg)-(Pf)(Pg)\big) & \textrm{under Rejective}\\
\end{cases}
\end{align*}
Here $\lambda = \lim_N n/N, A = \lim_N N^{-1}\sum_{i=1}^N \pi_i^{-1}$.
% and $\lambda (A-1) \ge 1-\lambda$ as noted above.
\end{corollary}
As we will see in the proofs, the covariance structure of the limit of $\sqrt{n}(\Prob_N^{\pi,H}-\Prob_N)$ is the same as that of 
\begin{align*}
f \mapsto \frac{1}{\sqrt{N}}\sum_{i=1}^N \bigg(\frac{\xi_i}{\pi_i}-1\bigg)(f(Y_i)-Pf)
\end{align*}
up to a factor of $\lambda$, which can be determined by the conditions of Proposition \ref{prop:cov_structure_cal}. Furthermore, the above corollary extends Theorem 4.2 of \cite{boistard2017functional}, again from the one-dimensional case to a general setting.

\begin{remark}
	Under (F3), since the harmonic mean is less than the arithmetic mean, we have $A^{-1} = \lim_N ( N^{-1} \sum_{i=1}^N \pi_i^{-1} )^{-1}  \le \lim_N ( N^{-1} \sum_{i=1}^N \pi_i ) = \lim_N \frac{n}{N} = \lambda$, where the next to last equality follows by computing the second moment and using (F3). It then follows that $\lambda (A-1) \geq 1 - \lambda$ under (F3).
\end{remark}

\subsection{Conditional limit theorems}

In this section, we consider conditional versions of the (global) uniform limit theorems. For clarity of presentation, following \cite{cheng2010bootstrap} and \cite{wellner1996bootstrapping}, we introduce the following notion:
\begin{definition}
	Let $\{\Delta_N\}_{N \in \N}$ be a sequence of random variables defined on $(\mathcal{S}_N\times \mathcal{X},\sigma(\mathcal{S}_N)\times \mathcal{A}, \Prob)$. We say that $\Delta_N$ is of order $\mathfrak{o}_{\Prob_d}(1)$ in $\Prob_{(Y,Z)}$-probability if for any $\epsilon,\delta>0$, we have $
	\Prob_{(Y,Z)} \big(\Prob_{d|(Y,Z)}\big(\abs{\Delta_N}>\epsilon\big)>\delta\big)\to 0 $ as $N \to \infty$. 
%	\begin{enumerate}
%		\item $\Delta_N$ is of order $o_{\Prob_d}(1)$ in $\Prob_{(Y,Z)}$-probability if for any $\epsilon,\delta>0$, we have $
%		\Prob_{(Y,Z)} \big(\Prob_{d|(Y,Z)}\big(\abs{\Delta_N}>\epsilon\big)>\delta\big)\to 0 $ as $N \to \infty$. 
%		\item 	$\Delta_N$ is of order $O_{\Prob_d}(1)$ in $\Prob_{(Y,Z)}$-probability if for any $\delta>0$, there exists some $M>0$ such that $
%		\Prob_{(Y,Z)} \big(\Prob_{d|(Y,Z)}\big(\abs{\Delta_N}>M\big)>\delta\big)\to 0 $
%		as $N \to \infty$. 
%	\end{enumerate}
\end{definition}

Below we establish conditional versions of Glivenko-Cantelli and Donsker theorems for $\Prob_N^\pi-\Prob_N$.
\begin{corollary}(Conditional Glivenko-Cantelli Theorem)\label{cor:conditional_ULLN}
	Suppose that (A1) and (A2-LLN) hold. If $\mathcal{F}$ is $P$-Glivenko-Cantelli, then 
	\begin{align*}
	\sup_{f \in \mathcal{F}}\abs{(\Prob_N^\pi-\Prob_N)(f)}=\mathfrak{o}_{\Prob_d}(1) \textrm{ in } \Prob_{(Y,Z)}\textrm{-probability}.
	\end{align*}
\end{corollary}

\begin{corollary}(Conditional Donsker Theorem)\label{cor:conditional_UCLT}
	Suppose that (A1) and (A2-CLT) hold, and that $\mathcal{F}$ is $P$-Donsker. Further suppose that the conditions in Proposition \ref{prop:cov_G_pi} hold, and that $n/N\to \lambda \in (0,1)$. Then 
	\begin{align*}
	\sqrt{n}(\Prob_N^\pi - \Prob_N)\rightsquigarrow \bar{\G}^\pi\textrm{ in }\ell^\infty(\mathcal{F})\textrm{ in } \Prob_{(Y,Z)}\textrm{-probability}.
	\end{align*}
	Here $\bar{\G}^\pi$ is a Gaussian process whose covariance structure is given in Corollary \ref{cor:P_N_center_Donsker}.
\end{corollary}

The precise meaning of the above conditional Donsker theorem is that $d_{\mathrm{BL},d}(\sqrt{n}(\Prob_N^\pi - \Prob_N), \bar{\G}^\pi)\equiv \sup_{H \in \mathrm{BL}_1(\ell^\infty(\mathcal{F}))} \abs{\E^\ast_{d|(Y,Z)} H\big(\sqrt{n}(\Prob_N^\pi - \Prob_N)\big) - \E H(\bar{\G}^\pi)}\to 0$ in $\Prob_{(Y,Z)}$-probability.

\section{Applications}\label{section:applications}

In this section, we apply the new tools developed in Section \ref{section:theory} in statistical problems including:
\begin{enumerate}
	\item $M$-estimation (or \emph{empirical risk minimization}) in a general non-parametric model;
	\item $Z$-estimation in a general semi-parametric model;
	\item frequentist theory for Bayes procedures, namely, theory of posterior contraction rates and Bernstein-von Mises type theorems,
\end{enumerate}
where the usual likelihood is replaced by the Horvitz-Thompson weighted likelihood. We will not consider the calibrated version of these problems for simplicity of exposition, given that the corresponding theory has been fully developed in Section \ref{section:theory}. These problems are not meant to be exhaustive; they are demonstrated as an illustration and a proof of concept of the new tools.

\subsection{$M$-estimation}
Consider the canonical \emph{empirical risk minimization} problem (or ``$M$-estimation'') based on weighted likelihood:
\begin{align}\label{def:ERM}
\hat{f}_N^\pi \equiv \arg\min_{f \in \mathcal{F}} \Prob_N^\pi f.
\end{align}
The quality of the estimator defined in (\ref{def:ERM}) is evaluated through the \emph{excess risk} of $\hat{f}_N^\pi$, denoted $\mathcal{E}_P(\hat{f}_N^\pi)$, where 
\begin{align*}
\mathcal{E}_P(f)\equiv Pf-\inf_{g \in \mathcal{F}} Pg,\quad f \in \mathcal{F}.
\end{align*}
The problem of studying excess risk of empirical risk minimizers under the usual empirical measure has been extensively studied in the 2000s; we only refer the reader to \cite{gine2006concentration,koltchinskii2006local,koltchinskii2008oracle} and references therein. Under the Horvitz-Thompson empirical measure, \cite{clemenccon2016learning} studied risk bounds for the binary classification problem under sampling designs that are close to the rejective sampling design. Our goal here will be a study of the excess risk for the $M$-estimator based on weighted likelihood as defined in (\ref{def:ERM}) for the general empirical risk minimization problem under general sampling designs.

To this end, let $\mathcal{F}_{\mathcal{E}}(\delta)\equiv \{f \in \mathcal{F}: \mathcal{E}_P(f)<\delta^2\}$, let $\rho_P:\mathcal{F}\times \mathcal{F}\to \R_{\geq 0}$ be such that $\rho_P^2(f,g)\geq P(f-g)^2-\big(P(f-g)\big)^2$, and $D(\delta)\equiv \sup_{f,g \in \mathcal{F}_{\mathcal{E}}(\delta)} \rho_P(f,g)$. Now we may prove the following theorem.
\begin{theorem}\label{thm:M_estimation}
Suppose (A1) holds. Suppose that there exist some $L>0,\kappa\geq 1$ such that $D(\delta)\leq L\cdot \delta^{1/\kappa}$, and that $\mathcal{F}$ is uniformly bounded and satisfies an entropy condition with exponent $\alpha \in (0,2)$. Then there exist some constants $\{C_i\}_{i=1}^3$ only depending on $\pi_0,L,\kappa,\alpha$ such that for any $s,t\geq 0$, with
\begin{align*}
r_N\geq C_1 N^{-\frac{\kappa}{4\kappa-2+\alpha}}+C_2 \bigg(\frac{s\vee t^2}{N}\bigg)^{\frac{\kappa}{4\kappa-2} },
\end{align*}
%then
%\begin{align*}
%\Prob\bigg(\sup_{f \in \mathcal{F}: \mathcal{E}_P(f)\geq r_N^2} \biggabs{\frac{\mathcal{E}_{\Prob_N^\pi}(f)}{\mathcal{E}_P(f)}-1}\geq 3/4\bigg)\leq \frac{C_3}{s}e^{-s/C_3}.
%\end{align*}
it holds that
\begin{align*}
\Prob\big(\mathcal{E}_{P}(\hat{f}_N^\pi)\geq r_N^2\big)\leq \frac{C_3}{s}e^{-s/C_3}+\Prob\bigg(\biggabs{\frac{1}{\sqrt{N}}\sum_{i=1}^N\bigg(\frac{\xi_i}{\pi_i}-1\bigg) }>t\bigg).
\end{align*}
\end{theorem}

As an illustration of Theorem \ref{thm:M_estimation}, we consider below two standard settings, regression and classification, similar to the development in \cite{gine2006concentration}. For simplicity of exposition, we also assume that (A2-CLT) holds.

\begin{example}[Bounded regression]
Let $\{(X_i,Y_i) \in \mathcal{X}\times [-1,1]\}_{i=1}^N$ denote the i.i.d. copies of the pairs consisting of covariates $X_i$ and responses $Y_i$. Our goal is to estimate the regression function $g_0(x)\equiv \E[Y|X=x]$ using the weighted least squares method:
\begin{align*}
\hat{g}_N^\pi \equiv \arg\min_{g \in \mathcal{G}} \sum_{i=1}^N \frac{\xi_i}{\pi_i} \big(Y_i-g(X_i)\big)^2,
\end{align*}
where $\mathcal{G}$ is a function class containing functions taking values in $[-1,1]$, and the weights $\{\xi_i,\pi_i\}$ may depend on auxiliary information $Z^{(N)}$. To apply Theorem \ref{thm:M_estimation}, let $\mathcal{F}\equiv \{f_g(x,y)\equiv (y-g(x))^2: g \in \mathcal{G}\}$. Then following the arguments in page 1208 of \cite{gine2006concentration}, we have $
\mathcal{E}_P(f_g)=\pnorm{g-g_0}{L_2(P)}^2$ and we may take $\kappa = 1$. If $\mathcal{G}$ satisfies an entropy condition with exponent $\alpha \in (0,2)$, it is easy to verify that the same holds for $\mathcal{F}$ and hence Theorem \ref{thm:M_estimation} yields
%a typical convergence rate $\pnorm{\hat{g}_N^\pi-g_0}{L_2(P)}^2= \mathcal{O}_{\mathbf{P}}(N^{-\frac{2}{2+\alpha}})$.
\begin{align*}
\pnorm{\hat{g}_N^\pi-g_0}{L_2(P)}^2= \mathcal{O}_{\mathbf{P}}\big(N^{-\frac{2}{2+\alpha}}\big),
\end{align*}
a very typical rate in the regression problem.
\end{example}

\begin{example}[Classification]
Let $\{(X_i,Y_i) \in \mathcal{X}\times \{0,1\}\}_{i=1}^N$ denote the i.i.d. copies of the pairs consisting of covariates $X_i$ and responses $Y_i$. A classifier $g: \mathcal{X} \to \{0,1\}$ over a class $\mathcal{G}$ has a generalized error $P(Y\neq g(X))$. The excess risk for a classifier $g$ over $\mathcal{G}$ under law $P$ is given by
\begin{align*}
\mathcal{E}_P(g)\equiv P(Y\neq g(X))-\inf_{g' \in \mathcal{G}} P(Y\neq g'(X)).
\end{align*}
It is known that for a given law $P$ on $(X,Y)$, the minimal generalized error is attained by a Bayes classifier $g_0(x)\equiv \bm{1}_{\eta(x)\geq 1/2}$ where $\eta(x)\equiv \E[Y|X=x]$, cf. \cite{devroye1996probabilistic}. In the setting of complex sampling design, it is natural to estimate $g_0$ by minimizing the weighted training error:
\begin{align*}
\hat{g}_N^\pi \equiv \arg\min_{g \in \mathcal{G}} \sum_{i=1}^N \frac{\xi_i}{\pi_i} \bm{1}_{Y_i\neq g(X_i)},
\end{align*}
where $g_0 \in \mathcal{G}$ is a collection of classifiers. To apply Theorem \ref{thm:M_estimation}, let $\mathcal{F}\equiv \{f_g\equiv\bm{1}_{y\neq g(x)}: g \in \mathcal{G}\}$. Suppose the following margin condition (cf. \cite{mammen1999smooth,tsybakov2004optimal}) holds for some $c>0,\kappa\geq 1$: for all $g \in \mathcal{G}$
\begin{align}
\mathcal{E}_P(g)\geq c \Pi^\kappa (g(X)\neq g_0(X)),
\end{align}
where $\Pi$ is the marginal law of $X$ under $P$. Following page 1212 of \cite{gine2006concentration}, we may choose $D(\delta)\lesssim \delta^{1/\kappa}$, and hence if the collection of classifiers $\mathcal{G}$ satisfies an entropy condition with exponent $\alpha \in (0,2)$, using $(f_{g_1}-f_{g_2})^2\leq (g_1-g_2)^2$, we see that $\mathcal{F}$ also satisfies the same entropy condition and hence 
\begin{align*}
P(Y\neq \hat{g}_N^\pi (X))-\inf_{g' \in \mathcal{G}} P(Y\neq g'(X))= \mathcal{O}_{\mathbf{P}}\big(N^{-\frac{\kappa}{2\kappa-1+\alpha/2}}\big),
\end{align*}
a very typical rate in the classification problem.
\end{example}

\subsection{$Z$-estimation}

The method of $Z$-estimation that produces estimators by finding those values of the parameters which zero out a set of estimating equations is well-understood by now under the usual empirical measure; see \cite{van2002semiparametric,van1996weak} for a comprehensive treatment. With the Horvitz-Thompson empirical measure, \cite{breslow2007weighted,breslow2008theorem,saegusa2018large,saegusa2013weighted} considered weighted likelihood estimation under stratified sampling designs, both with and without overlaps. The goal of this section is to give a unified theoretical treatment for the $Z$-estimation problem under general sampling designs.

Let $\hat{\theta}_N^\pi \in \Theta$ solve the (possibly infinite-dimensional) estimating equations based on weighted likelihood:
\begin{align*}
\Prob_N^\pi \psi_{\hat{\theta}_N^\pi,h} = 0,\quad \textrm{ for all }h \in \mathcal{H},
\end{align*}
while the `truth' $\theta_0 \in \Theta$ solves the population equations
\begin{align*}
P \psi_{\theta_0,h} = 0 ,\quad \textrm{ for all } h \in \mathcal{H}.
\end{align*}
Let $\Psi_N,\Psi:\Theta \to \ell^\infty(\mathcal{H})$ be given by $\Psi_N(\theta)(h)\equiv \Prob_N^\pi \psi_{\theta,h}$ and $\Psi(\theta)(h)\equiv P\psi_{\theta,h}$. We assume that $\mathcal{H}$ is countable without loss of generality.

\begin{theorem}\label{thm:Z_est}
Suppose that (A1) and (A2-CLT) hold, and that the following conditions hold.
\begin{enumerate}
\item[(Z1)] The map $\Psi$ is Fr\'echet differentiable at $\theta_0$ with a continuously invertible derivative  $\dot{\Psi}_{\theta_0}$.
\item[(Z2)] The stochastic equi-continuity condition holds:
\begin{align*}
\pnorm{\G_N(\psi_{\hat{\theta}_N^\pi,h}-\psi_{\theta_0,h})}{\mathcal{H}}= \mathfrak{o}_{\mathbf{P}}\big(1+\sqrt{N}\pnorm{\hat{\theta}_N^\pi-\theta_0}{}\big)
\end{align*}
and $\{\psi_{\theta_0,h}: h \in \mathcal{H}\}$ is a $P$-Glivenko-Cantelli class.
\end{enumerate}
If $\hat{\theta}_N^\pi\to_{\mathbf{P}}\theta_0$, then
\begin{align*}
\sqrt{N}(\hat{\theta}_N^\pi-\theta_0) = -\dot{\Psi}_{\theta_0}^{-1} \G_N^\pi \psi_{\theta_0,\cdot}+\mathfrak{o}_{\mathbf{P}}(1).
\end{align*}
\end{theorem}

This theorem is comparable to the standard $Z$-Theorem 3.3.1 in \cite{van1996weak}, but here we work in the context of $Z$-estimation under weighted likelihood. Note that our conditions are are almost identical to the standard $Z$-Theorem, many examples for which Theorem \ref{thm:Z_est} applies can be found in Section 3.3 of \cite{van1996weak} (see also \cite{van2002semiparametric,van1995efficiency}). In particular, (Z2) is imposed for the usual empirical process $\G_N$, and can be easily checked if a Donsker property for the class $\{\psi_{\theta,h}-\psi_{\theta_0,h}:\pnorm{\theta-\theta_0}{}\leq \delta, h \in \mathcal{H}\}$ holds. We omit these details here.

Now consider estimation of a finite-dimensional parameter in the presence of an infinite-dimensional nuisance parameter, i.e. estimation in a semi-parametric model. Following \cite{cheng2010bootstrap,ma2005robust}, we use the following general semi-parametric framework: Consider a model $\{P_{\theta,\eta}: (\theta,\eta)\in \R^d \times \mathcal{H}\}$, where $\mathcal{H}$ is an infinite dimensional Hilbert space with norm $\pnorm{\cdot}{\mathcal{H}}$. Suppose that the true parameter is $(\theta_0,\eta_0)$. An estimator $(\hat{\theta}_N^\pi,\hat{\eta}_N^\pi)$ of $(\theta_0,\eta_0)$ usually takes the form
\begin{align}\label{def:par_semi_generic}
(\hat{\theta}_N^\pi,\hat{\eta}_N^\pi):=\arg\sup \Prob_N^\pi m_{\theta,\eta},
\end{align}
where $m_{\theta,\eta}$ is often the log likelihood function (for $n=1$). However here we will work with a more general $Z$-estimation framework.

For any fixed $\eta \in \mathcal{H}$, let $\eta(t)$ be a smooth curve at $t=0$ with $\eta(0)=\eta$ and $a=(\partial/\partial t) \eta(t)\lvert_{t=0}$ for some $a \in \mathcal{H}$. Denote $\mathcal{A}\subset \mathcal{H}$ the collection for all such admissible $a$'s. Now let $m_\theta(\theta,\eta)=\partial_{\theta}m(\theta,\eta) \in \R^d$, $m_\eta(\theta,\eta)[a]=(\partial/\partial t)m(\theta,\eta(t))|_{t=0}$ with $\partial_t \eta(t)|_{t=0}=a \in \mathcal{A}$. The second derivatives can be defined in a similar fashion. Suppose further the following orthogonality condition hold: there exists $A^\ast=(a_1^\ast,\ldots,a_d^\ast) \in \mathcal{A}^d$ so that for any $A \in \mathcal{A}^d$, it holds that
\begin{align}\label{cond:semipara_orthogonal_0}
P_{\theta_0,\eta_0}\big(m_{\theta\eta}(\theta_0,\eta_0)[A]-m_{\eta\eta}[A^\ast][A]\big)=0.
\end{align}
This condition is commonly adopted in semi-parametric literature to handle the case when nuisance parameter is not $\sqrt{n}$-estimable; see, e.g., Condition 2, page 555 in \cite{huang1996efficient}\footnote{See also condition A3 in \cite{wellner2007two}, page 2138; condition (4) in \cite{ma2005robust}, page 196; condition (4) in \cite{cheng2010bootstrap}, page 2887.}. 

Define the \emph{efficient score function} $\tilde{m}(\theta,\eta)=m_\theta(\theta,\eta)-m_\eta(\theta,\eta)[A^\ast]$ (since if $m$ is the log likelihood function, $\tilde{m}$ typically becomes the efficient score function). Then (\ref{cond:semipara_orthogonal_0}) can be rewritten as following: for any $A \in \mathcal{A}^d$,
\begin{align}\label{cond:semipara_orthogonal}
P_{\theta_0,\eta_0} \tilde{m}_\eta (\theta_0,\eta_0)[A]=0.
\end{align}
We assume that the true parameter $(\theta_0,\eta_0)$ zeros out the population estimating equation:
\begin{align}\label{cond:no_bias}
P_{\theta_0,\eta_0} \tilde{m}(\theta_0,\eta_0) = 0.
\end{align}
To allow some flexibility in the framework, the estimators $(\hat{\theta}_N^\pi,\hat{\eta}_N^\pi)$ are assumed to approximately zero out the Horvitz-Thompson empirical estimating equation: 
\begin{align}\label{def:par_semi}
\Prob_N^\pi \tilde{m}(\hat{\theta}_N^\pi,\hat{\eta}_N^\pi) = \mathfrak{o}_{\mathbf{P}}(N^{-1/2}).
\end{align}
It is easy to see that the above condition is satisfied if (\ref{def:par_semi_generic}) holds. Note here our general condition also includes the case where $\hat{\eta}_N^\pi$ may depend on $\hat{\theta}_N^\pi$, e.g. profile likelihood estimation. 
%When $m$ is the log likelihood function, $\tilde{m}$ typically becomes the efficient score function.
%i.e. the projection of the score of $m_\theta$ onto the orthocomplement of $m_\eta$ in $L_2(P)$. 

\begin{theorem}\label{thm:semipara}
Suppose that (A1) holds, and that (\ref{cond:semipara_orthogonal})-(\ref{def:par_semi}) hold. Further assume the following conditions.
\begin{enumerate}
\item[(S1)] The matrix $
I_{\theta_0,\eta_0}\equiv -P_{\theta_0,\eta_0}\tilde{m}_\theta(\theta_0,\eta_0)\in\R^{d\times d}$
is non-singular.
\item[(S2)] $\pnorm{\hat{\theta}_N^\pi-\theta_0}{}\vee \pnorm{\hat{\eta}_N^\pi-\eta_0}{\mathcal{H}}= \mathcal{O}_{\mathbf{P}}(N^{-\beta})$ holds for some $\beta>1/4$.
\item[(S3)] The model is smooth in the sense that 
\begin{align*}
& \big\lVert P_{\theta_0,\eta_0}\big(\tilde{m}(\theta,\eta)-\tilde{m}(\theta_0,\eta_0)- \tilde{m}_\theta(\theta_0,\eta_0)(\theta-\theta_0)\big)\big\lVert_{}\\
&\qquad=\mathcal{O}\big(\pnorm{\theta-\theta_0}{}^2\vee \pnorm{\eta-\eta_0}{\mathcal{H}}^2\big)
\end{align*}
holds for $(\theta,\eta)$ close enough to $(\theta_0,\eta_0)$.
\item[(S4)] For any $C>0$,
\begin{align*}
\sup_{\pnorm{\theta-\theta_0}{}\vee \pnorm{\eta-\eta_0}{\mathcal{H}}\leq CN^{-\beta}}\abs{\G_N\big(\tilde{m}(\theta,\eta)-\tilde{m}(\theta_0,\eta_0)\big)}&=\mathfrak{o}_{\mathbf{P}}(1).
\end{align*}
\end{enumerate}
Then
\begin{align*}
\sqrt{N}(\hat{\theta}_N^\pi-\theta_0) = I_{\theta_0,\eta_0}^{-1} \G_N^\pi \tilde{m}(\theta_0,\eta_0)+\mathfrak{o}_{\mathbf{P}}(1).
\end{align*}
\end{theorem}

Conditions (S1)-(S4) are all standard assumptions in semi-parametric literature, and can be verified in numerous models, including the Cox model with right censored/current status data, partially linear model, panel count data (with covariates) etc. Here we only consider the partially linear model; detailed verifications for other models can be found in, e.g.  \cite{cheng2010bootstrap,ma2005robust,saegusa2018large,saegusa2013weighted,wellner2007two}.

\begin{example}[Partially linear model]
Consider the following model
\begin{align*}
Y_i = X_i^\top \theta_0 + f_0(W_i)+e_i,\quad  i=1,\ldots, N,
\end{align*}
where $Y_i$'s are the responses, $\{(X_i,W_i)\in [-1,1]^d\times [0,1] \}$'s are i.i.d. covariates, and $e_i$'s are i.i.d. normal errors independent of the covariates. The `true signal' $\theta_0 \in \R^d$ and $f_0:[0,1]\to \R$ is a `smooth' function. For ease of exposition we will consider the parameter space $
\Xi \equiv\{(\theta,f): \pnorm{\theta}{1} \leq 1, \pnorm{f}{\infty}\leq 1, J(f)\leq M\}$ for some $M>0$, and here $J^2(f):=\int_0^1 (f^{''}(t))^2\ \d{t}$. Now with $\lambda_N\asymp N^{-2/5}$, let
\begin{align}\label{def:partial_linear_par}
(\hat{\theta}_N^\pi,\hat{f}_N^\pi) :=\arg\min_{(\theta,f) \in \Xi} \bigg[\Prob_N^\pi \big(Y-X^\top \theta-f(W)\big)^2
+ \lambda_N^2 J^2(f)\bigg].
\end{align}
To put the model into our framework, let $
m(\theta,f):=-(y-x^\top \theta-f(w))^2$. Then for any admissible $a,b$, we have 
\begin{align*}
&m_\theta(\theta,f)= 2x(y-x^\top\theta-f(w)),& &m_f(\theta,f)[a] = 2a(w) (y-x^\top\theta -f(w)),\\
&m_{\theta f}(\theta,f)[b] = -2xb(w),& & m_{ff}(\theta,f)[a][b] = -2a(w)b(w).
\end{align*}
% m_\theta(\theta,f)= 2x(y-x^\top\theta-f(w))$, $
%m_f(\theta,f)[a] = 2a(w) (y-x^\top\theta -f(w))$, $
% m_{\theta f}(\theta,f)[b] = -2xb(w)$ and $
%m_{ff}(\theta,f)[a][b] = -2a(w)b(w) $.
Now let $
A^\ast(W) =\E[X|W] \in \R^d$. Then a direct calculation verifies (\ref{cond:semipara_orthogonal_0}). Thus we can take
\begin{align}\label{def:mtilde_par_linear}
\tilde{m}(\theta,f) = 2(y-x^\top \theta-f(w)) (x-\E[X|W=w]).
\end{align}
(\ref{cond:no_bias}) is immediately verified; (\ref{def:par_semi}) can also be verified by taking partial derivatives in the definition (\ref{def:partial_linear_par}) and noting that $\lambda_N^2=\mathfrak{o}(N^{-1/2})$. Now we verify (S1)-(S4). (S1) will be satisfied if the matrix $ I_{\theta_0,\eta_0}\equiv 2\E \big[(X-\E[X|W])X^\top]=2\E \big[(X-\E[X|W])^{\otimes 2}]$ is non-singular. (S2) can be verified with $\beta = 2/5$ along the lines of Lemma 25.88 in \cite{van2000asymptotic} with the tools developed in Section \ref{section:theory}. (S3) is trivially satisfied since $\tilde{m}$ is linear in $\theta$ and $f$. (S4) is also easy to verify. Hence we have shown that under the same conditions as in Lemma 25.88 of \cite{van2000asymptotic},
\begin{align*}
\sqrt{N}(\hat{\theta}_N^\pi-\theta_0) = I_{\theta_0,\eta_0}^{-1} \G_N^\pi \tilde{m}(\theta_0,\eta_0)+\mathfrak{o}_{\mathbf{P}}(1).
\end{align*}

\end{example}

\subsection{Frequentist theory for Bayesian procedures}

Suppose the i.i.d. super-population variables of interest $\{Y_i\}_{i=1}^N$ have law $P_{f_0}$ where $f_0$ belongs to a statistical model $\mathcal{F}$ and $\{P_f\}_{f \in \mathcal{F}}$ is dominated by a $\sigma$-finite measure $\mu$. A Bayesian approach assigns a prior $\Pi_N$ on the model $\mathcal{F}$ and makes estimation/inference based on the posterior distribution. In the case where all the super-population $\{Y_i\}_{i=1}^N$ are available, by Bayes' formula, the posterior distribution, i.e. a random measure on $\mathcal{F}$, is defined as follows: for a measurable subset $B \subset \mathcal{F}$, 
\begin{align}\label{eqn:post_dist_usual}
\Pi_N (B|Y^{(N)})& \equiv \frac{\int_B  \prod_{i=1}^N p_f(Y_i)\ \d{\Pi_N(f)}}{\int \prod_{i=1}^N p_f(Y_i)\ \d{\Pi_N(f)}}=\frac{\int_B  \exp \big(N\Prob_N \log p_f\big)\ \d{\Pi_N(f)}}{\int \exp \big(N\Prob_N \log p_f\big)\ \d{\Pi_N(f)}},
\end{align}
where $p_f(\cdot)$ denotes the probability density function of $P_f$ with respect to the dominating measure $\mu$.

In the current super-population setup with complex sampling designs, we may naturally replace the usual empirical measure $\Prob_N$ in (\ref{eqn:post_dist_usual}) by the Horvitz-Thompson empirical measure $\Prob_N^\pi$ to define the \emph{posterior distribution with weighted likelihood} as follows:  for a measurable subset $B \subset \mathcal{F}$, 
\begin{align}\label{eqn:post_dist}
\Pi_N^\pi (B|D^{(N)})& \equiv \frac{\int_B  \prod_{i=1}^N p_f(Y_i)^{\xi_i/\pi_i}\ \d{\Pi_N(f)}}{\int \prod_{i=1}^N p_f(Y_i)^{\xi_i/\pi_i}\ \d{\Pi_N(f)}}=\frac{\int_B  \exp \big(N\Prob_N^\pi \log p_f\big)\ \d{\Pi_N(f)}}{\int \exp \big(N\Prob_N^\pi \log p_f\big)\ \d{\Pi_N(f)}}.
\end{align}
Recall here $D^{(N)}\equiv (Y^{(N)},Z^{(N)},\xi^{(N)},\pi^{(N)})$. As we will see below, one particular advantage of the posterior distribution with weighted likelihood defined above is that we may obtain a complete frequentist theory for Bayes procedures analogous to that based on observing the whole super-population $\{Y_i\}_{i=1}^N$.

We say that the posterior distribution with weighted likelihood, namely $\Pi_N^\pi (\cdot|D^{(N)})$, contracts at a rate $\delta_N$ with respect to a metric $d$ if
\begin{align*}
&P_{f_0} \Pi_N^\pi \big(f \in \mathcal{F}: d^2(f,f_0)>L_N \delta_N^2\big \lvert D^{(N)}\big)\to 0
\end{align*}
for any $L_N\to \infty$. 

Our first goal in this section is to develop some useful results in deriving such posterior contraction rates for the posterior distribution using weighted likelihood. We will use (essentially the same) machinery developed in \cite{han2017bayes} (which we find easier to adapt in the current context than the standard machinery \cite{ghosal2000convergence,ghosal2007convergence}). For some $v>0,c\in [0,\infty)$ let
\begin{align*}
\psi_{v,c}(\lambda)=v\lambda^2\cdot \bm{1}_{\abs{\lambda}\leq 1/c}+ \infty\cdot \bm{1}_{\abs{\lambda}>1/c}
\end{align*}
denote the local quadratic function. 

\begin{theorem}\label{thm:bayes}
Suppose (A1) holds and the following conditions hold:
\begin{enumerate}
\item[(B1)] (Local Gaussianity condition) There exist some constants $c_1>0$ and $\kappa=(\kappa_g,\kappa_\Gamma) \in (0,\infty)\times [0,\infty)$ such that for all $n \in \N$, and $f_0,f_1 \in \mathcal{F}$, 
\begin{align*}
P_{f_0} \exp\left[\lambda \left(\log \frac{p_{f_0}}{p_{f_1}}-P_{f_0}\log \frac{p_{f_0}}{p_{f_1}}\right)\right]\leq c_1 \exp\left[\psi_{\kappa_g d^2(f_0,f_1),\kappa_\Gamma }(\lambda)\right]
\end{align*}
Here  $d:\mathcal{F}\times \mathcal{F}\to \R_{\geq 0}$ is a symmetric function satisfying
\begin{align*}
c_2\cdot  d^2(f_0,f_1)\leq P_{f_0}\log \frac{p_{f_0}}{p_{f_1}}\leq c_3 \cdot d^2(f_0,f_1)
\end{align*}
for some constants $c_2,c_3>0$.
\item[(B2)] (Local entropy condition) There exist some $\{\delta_N\}_{N\in \N}$ such that
\begin{align*}
1+\sup_{\epsilon>\delta_{N}}\log \mathcal{N}\big(c_5 \epsilon , \{f \in \mathcal{F}: d(f,f_0)\leq 2\epsilon\}, d\big)\leq c_4 N \delta_{N}^2
\end{align*}
where $c_4 \in (0,1),c_5 \in (0,1/4)$ depend on $\{c_i\}_{i=1}^3$.
\item[(B3)] (Prior mass condition) For all $j \in \N$,
\begin{align*}
\frac{\Pi_N\big(\{f\in \mathcal{F}: j\delta_N<d(f,f_0)\leq (j+1) \delta_N\}\big)}{\Pi_N\big(d(f,f_0)\leq\delta_N\big)}\leq \exp(c_6j^2N\delta_{N}^2),
\end{align*}
where $c_6>0$ is a small enough constant depending on $\{c_i\}_{i=1}^3$.
\end{enumerate}
Then
\begin{align*}
P_{f_0} \Pi_N^\pi \big(f \in \mathcal{F}: d^2(f,f_0)>C_1 \delta_N^2\big \lvert D^{(N)} \big)\leq C_2 \exp(-N\delta_N^2/C_2).
\end{align*}
%and with $\hat{f}_N^\pi\equiv P_{f_0} \Pi_N^\pi(f|D^{(N)})$ denoting the posterior mean,
%\begin{align*}
%P_{f_0} d^2(\hat{f}_N^\pi,f_0)\leq C_3 \delta_N^2.
%\end{align*}
Here $C_1,C_2>0$ only depend on $\{c_i\}_{i=1}^3$ and $\kappa$.
\end{theorem}

The local Gaussianity condition (B1) can be easily verified in a wide range of experiments including regression/density estimation/Gaussian autoregression/Gaussian time series/covariance matrix estimation, etc. (B2)-(B3) are standard conditions in the literature. In particular, (B3) allows the exact $\sqrt{N}$ parametric posterior contraction rate, which will be useful below. It is also possible to consider hierarchical priors to formulate a similar theorem as in \cite{han2017bayes}---in essence all examples therein can be applied here (except for regression where random design instead of fixed design is needed to maintain the i.i.d. property of the super-population $\{Y_i\}_{i=1}^N$). Although we refer the readers to \cite{han2017bayes} for more details and examples, we will demonstrate one example below for the convenience of the reader.
%The local entropy condition and the prior mass condition are essentially the same as in \cite{ghosal2000convergence,ghosal2007convergence}, and hence many examples therein (and in the literature) can be directly applied.
\begin{example}
Consider the covariance matrix estimation problem: suppose $Y_1,\ldots,Y_N \in \R^d$ are i.i.d. observations from $\mathcal{N}_d(0,\Sigma_0)$ where $\Sigma_0 \in \mathscr{S}_d(L)$, the set of $d\times d$ covariance matrices whose minimal and maximal eigenvalues are bounded by $L^{-1}$ and $L$, respectively. The covariance matrix is modeled by the sparse factor model $\mathfrak{M}\equiv \cup_{(k,s)\in \N^2 }\mathfrak{M}_{(k,s)}$ where $\mathfrak{M}_{(k,s)}\equiv \{\Sigma=\Lambda \Lambda^\top +I : \Lambda \in \mathscr{R}_{(k,s)}(L)\}$ with $\mathscr{R}_{(k,s)}(L)\equiv \{\Lambda\in \R^{d\times k}, \Lambda_{\cdot j} \in B_0(s),  \abs{\sigma_j(\Lambda)}\leq L^{1/2},\textrm{ for all } 1\leq j\leq k\}$. 

Suppose we use a hierarchical prior $\Pi_N=\sum_{(k,s) \in \N^2} \lambda_{N}((k,s)) \Pi_{N,(k,s)}$ with the same model selection priors $\{\lambda_{N}((k,s))\}_{(k,s) \in \N^2}$ and the sieve priors $\{\Pi_{N,(k,s)}\}_{(k,s) \in \N^2}$ specified as in \cite{han2017bayes}, then
\begin{align*}
 P_{\Sigma_0} \Pi_N^\pi\bigg(\Sigma \in \mathfrak{M}: \pnorm{\Sigma-\Sigma_0}{F}^2> C_1 \frac{ks\log(e d)}{N}\big\lvert D^{(N)}\bigg)  \leq C_2 \exp\left(-ks(\log ed)/C_2\right).
\end{align*}
Here $\pnorm{\cdot}{F}$ denotes the matrix Frobenius norm.

\end{example}

Next we will be interested in a more precise limiting distribution of the posterior distribution with weighted likelihood, i.e. a Bernstein-von Mises type theorem. To this end, we work with a finite-dimensional model $ \Theta$ being a compact subset of $\R^d$. Let $\theta_0 \in \Theta$, an interior point of $\Theta$, be the true signal. Let $\mathcal{N}_{\mu,\Sigma}$ denote the $d$-dimensional normal distribution with mean $\mu$ and covariance matrix $\Sigma$. 

\begin{theorem}\label{thm:bvm}
Suppose that (A1) and (A2-CLT) hold. Further assume the following conditions.
\begin{enumerate}
\item[(Bv1)] (Experiment) The map $\theta \mapsto \log p_{\theta}(x)=\ell_\theta(x)$ is differentiable at $\theta_0$ for all $x$ with derivative $\dot{\ell}_{\theta_0}(x)$, and for $\theta_1,\theta_2$ close enough to $\theta$,
\begin{align*}
\abs{\ell_{\theta_1}(x)-\ell_{\theta_2}(x)}\leq m(x)\pnorm{\theta_1-\theta_2}{}
\end{align*}
holds for some $P_{\theta_0}$-square integrable function $m$. Furthermore, the log-likelihood ratio $\big\{\log \frac{p_\theta}{p_{\theta_0}}\big\}_{\theta \in \Theta}$ satisfies the local Gaussianity condition, and is twice differentiable under $P_{\theta_0}$ with a non-singular Hessian $I_{\theta_0}$: for $\theta$ close enough to $\theta_0$,
\begin{align*}
P_{\theta_0}\log \frac{p_{\theta_0}}{p_\theta} = \frac{1}{2} (\theta-\theta_0)I_{\theta_0} (\theta-\theta_0)+\mathfrak{o}\big(\pnorm{\theta-\theta_0}{}^2\big).
\end{align*}
\item[(Bv2)] (Prior) The prior $\Pi$ has a Lebesgue density bounded away from $0$ and $\infty$ on $\Theta$.
\end{enumerate}
Then the posterior distribution with weighted likelihood $\Pi_N^\pi$ converges to a sequence of normal distributions in the total variational distance:
\begin{align*}
\sup_{B}\bigabs{\Pi_N^\pi \big( \sqrt{N}(\theta-\theta_0) \in B\big\lvert D^{(N)}\big)-\mathcal{N}_{I_{\theta_0}^{-1}\G_N^\pi \dot{\ell}_{\theta_0}, I_{\theta_0}^{-1}}(B) } = \mathfrak{o}_{\mathbf{P}}(1).
\end{align*}
\end{theorem}

Note that in finite-dimensional problems, the efficient score $\tilde{m}$ in Theorem \ref{thm:semipara} can usually be taken as $\dot{\ell}_{\theta_0}$, then under the regularity conditions as in Theorem \ref{thm:semipara}, we have the usual interpretation of the Bernstein-von Mises theorem in our context of weighted likelihood estimation: the sequence of posterior distributions with weighted likelihood resembles that of progressively sharpened normal distributions centered at the maximum weighted likelihood estimator $\hat{\theta}_N^\pi$:
\begin{align*}
\sup_{B}\bigabs{\Pi_N^\pi \big( \theta \in B\big\lvert D^{(N)}\big)-\mathcal{N}_{\hat{\theta}_N^\pi,N^{-1} I_{\theta_0}^{-1}}(B) } = \mathfrak{o}_{\mathbf{P}}(1).
\end{align*}

\section{Proofs for Section \ref{section:theory}}\label{section:proof_theory}

%\subsection{Proofs of Theorems \ref{thm:ULLN_HT} and \ref{thm:weak_convergence_HT}}

%Note that in the Proposition \ref{prop:multiplier_ineq}, the assumption that $\xi_1,\ldots,\xi_N$ being non-negative is crucial. Counter-examples for which the conclusion of Proposition \ref{prop:multiplier_ineq} when $\xi_1,\ldots,\xi_N$ can be negative are constructed in \cite{han2017sharp}.

In this section we present proofs of the main steps for the results in Section \ref{section:theory}. Many intermediate technical results will be deferred to Section \ref{section:remaining_proof}.
\begin{proof}[Proof of Theorem \ref{thm:ULLN_HT}]
Note that
\begin{align}\label{ineq:ULLN_HT_0}
\big(\Prob_N^{\pi}-P\big)(f)&= \frac{1}{N}\sum_{i=1}^N \frac{\xi_i}{\pi_i} \big(f(Y_i)-Pf\big)+Pf\cdot \frac{1}{N}\sum_{i=1}^N \left(\frac{\xi_i}{\pi_i}-1\right),
\end{align}
and we will handle two terms in (\ref{ineq:ULLN_HT_0}) separately. 

We handle the first term in (\ref{ineq:ULLN_HT_0}) by Proposition \ref{prop:multiplier_ineq}:
\begin{align*}
\E \sup_{f \in \mathcal{F}} \biggabs{\frac{1}{N}\sum_{i=1}^N \frac{\xi_i}{\pi_i} \big(f(Y_i)-Pf\big) }\lesssim N^{-1} \E \sup_{f \in \mathcal{F}} \biggabs{\sum_{i=1}^N \big(f(Y_i)-Pf\big)} \to 0,
\end{align*}
where the convergence follows by the fact that $\mathcal{F}$ is $P$-Glivenko-Cantelli, and the fact that the sequence $\{N^{-1}\sup_{f \in \mathcal{F}} \abs{\sum_{i=1}^N \big(f(Y_i)-Pf\big)} \}_{N=1}^\infty $ is a reversed submartingale (with respect to the permutation filtration, cf. Lemma 2.4.5 of \cite{van1996weak}) and hence convergence in probability is equivalent to convergence in expectation. The second term in (\ref{ineq:ULLN_HT_0}) also vanishes as $N \to \infty$ by the assumptions. 
%Finally, the almost sure convergence follows from the reversed sub-martingale property of the Horvitz-Thompson empirical measure, which can be established using similar arguments as in Lemma 2.4.5 of \cite{van1996weak} (by noting that the important structure therein is the average of the samples).
\end{proof}

Let
\begin{align*}
\tilde{\G}_N^\pi(f)\equiv \frac{1}{\sqrt{N}}\sum_{i=1}^N \frac{\xi_i}{\pi_i}\big(f(Y_i)-Pf\big).
\end{align*}
Then
\begin{align}\label{ineq:decomposition_Gtilde}
\G_{N}^\pi (f) = \tilde{\G}_N^\pi(f)+ Pf\cdot \frac{1}{\sqrt{N}}\sum_{i=1}^N \left(\frac{\xi_i}{\pi_i}-1\right).
\end{align}

\begin{proof}[Proof of Theorem \ref{thm:weak_convergence_HT}]
%\noindent (\emph{Sufficiency}).
We only need to prove asymptotic equi-continuity. Let $\mathcal{F}_{\delta}\equiv \{f-g: f \in \mathcal{F},g\in \mathcal{F}, \pnorm{f-g}{L_2(P)}\leq \delta\}$. 
%Note that
%\begin{align}\label{ineq:CLT_HT_0}
%\G_N^{\pi}(f) &= \frac{1}{\sqrt{N}}\sum_{i=1}^N \frac{\xi_i}{\pi_i} \big(f(Y_i)-Pf\big)+Pf\cdot \frac{1}{\sqrt{N}}\sum_{i=1}^N \left(\frac{\xi_i}{\pi_i}-1\right).
%\end{align}
We only need to assert asymptotic equi-continuity for the two terms in (\ref{ineq:decomposition_Gtilde}). 

Fix any $\delta_N\to 0$. For the first term in (\ref{ineq:decomposition_Gtilde}), using Proposition \ref{prop:multiplier_ineq} we have
\begin{align}\label{ineq:CLT_HT_4}
&\E \sup_{f \in \mathcal{F}_{\delta_N}} \biggabs{\frac{1}{\sqrt{N}}\sum_{i=1}^N \frac{\xi_i}{\pi_i} \big(f(Y_i)-Pf\big) }\lesssim  N^{-1/2} \E \sup_{f \in \mathcal{F}_{\delta_N}} \biggabs{\sum_{i=1}^N \big(f(Y_i)-Pf\big)}\to 0.\nonumber
\end{align}
Here in the above display we used that $\mathcal{F}$ is $P$-Donsker and Lemma  2.3.11 of \cite{van1996weak} to move from asymptotic equi-continuity in probability to in expectation.

For the second term in (\ref{ineq:decomposition_Gtilde}), we simply use (A2-CLT) and the fact that $\sup_{f \in \mathcal{F}_{\delta_N}}\abs{Pf}\leq \delta_N\to 0$. This completes the proof.
\end{proof}

\begin{proof}[Proof of Proposition \ref{prop:cov_G_pi}]
	We check the covariance structure by means of the Cram\'er-Wold device. For any $\bm{f}=(f_\ell)_{\ell=1}^k \in \mathcal{F}^{\otimes k}$ and $\bm{a}=(a_\ell)_{\ell=1}^k$, let $f\equiv \sum_{\ell=1}^k a_\ell f_\ell= \bm{a}^\top \bm{f}$. Note that
	\begin{align*}
	\G_N^\pi(f)&\equiv \sqrt{N}(\Prob_N^\pi f-\Prob_N f)+\G_N(f)\\
	& = \sqrt{NS_N^2} \cdot \frac{1}{S_N}\bigg(\frac{1}{N}\sum_{i=1}^N \frac{\xi_i}{\pi_i}f(Y_i)-\frac{1}{N}\sum_{i=1}^N f(Y_i)\bigg)+\G_N(f)\\
	&\rightsquigarrow \mathcal{N} \big(0, \bm{a}^\top\big((1+\mu_{\pi 1}) P[\bm{f}\bm{f}^\top]-(1-\mu_{\pi2}) (P\bm{f})(P\bm{f}^\top)\big)\bm{a}\big),
	\end{align*}
	where the convergence in distribution follows by Lemma \ref{lem:sum_cond_CLT} and the fact that
	\begin{align*}
	N S_N^2 \to_{\Prob_{(Y,Z)}} \lim_N \E [NS_N^2]&=\mu_{\pi1} Pf^2 + \mu_{\pi 2} (Pf)^2\\
	& = \bm{a}^\top\big( \mu_{\pi 1} P[\bm{f}\bm{f}^\top] + \mu_{\pi 2}  (P\bm{f})(P\bm{f}^\top) \big)\bm{a}.
	\end{align*}
	Here the convergence in probability in the above display follows from the same arguments as in Lemma B.1 of \cite{boistard2017functional} by calculating the variance of $NS_N^2$ with the help of (F3).
\end{proof}

%\subsection{Proofs of Theorems \ref{thm:local_asymp_moduli}-\ref{thm:calibrate_HT}}

\begin{proof}[Proof of Theorem \ref{thm:local_asymp_moduli}]
We will use Proposition \ref{prop:deviation_normalized_process} to prove the theorem. Take $q=2$, $r_N= N^{-1/(\alpha+2)}$ and $\delta_N=\mathfrak{o}(1/\log^{1/\alpha} N)$ therein. Since $\mathcal{F}$ satisfies an entropy condition with exponent $\alpha \in (0,2)$, by the local maximal inequalities in Lemma \ref{lem:local_maximal_ineq}, it follows that
\begin{align*}
\beta_{N,2}(r_N,\delta_N)= \max_{1\leq j\leq l} \frac{\E \pnorm{\G_N}{\mathcal{F}(r_N2^{j-1},r_N2^j)} }{\omega_\alpha(r_N2^j)}\lesssim \max_{1\leq j\leq l} \frac{ (r_N 2^j)^{1-\frac{\alpha}{2}} }{\omega_\alpha(r_N2^j)}\leq C_1.
\end{align*}
Choosing $s_j\equiv 3\log N$, we have that
\begin{align*}
\tau_{N,2}(r_N,\delta_N,\bm{s})&\asymp \max_{1\leq j\leq l} \frac{r_N 2^j \sqrt{\log N} + \log N/\sqrt{N}}{r_N^{1-\frac{\alpha}{2}} 2^{j(1-\frac{\alpha}{2})} }\\
&\lesssim \delta_N^{\alpha/2} \sqrt{\log N} +\frac{\log N}{\sqrt{N} r_N^{1-\frac{\alpha}{2}} }=\mathfrak{o}(1).
\end{align*}
Using Proposition \ref{prop:deviation_normalized_process} we see that
\begin{align*}
\sup_{f \in \mathcal{F}: r_N^2<Pf^2\leq \delta_N^2} \frac{\abs{\tilde{\G}_N^\pi(f)}}{\omega_\alpha(\sigma_P f)}=\mathcal{O}_{\mathbf{P}}(1).
\end{align*}
On the other hand, the second term in (\ref{ineq:decomposition_Gtilde}) (divided by $\omega_\alpha(\sigma_P f)$) is $\mathfrak{o}_{\mathbf{P}}(1)$ by the assumption (A2-CLT).
\end{proof}

\begin{proof}[Proof of Theorem \ref{thm:ratio_HT_rate}]
We first prove the first claim. Take $\phi(x)=x$. Note that this time with $r_N\gtrsim N^{-1/(\alpha+2)}$,
\begin{align*}
\beta_{N,q}\lesssim \max_{1\leq j\leq l}\frac{ (r_N q^j)^{1-\frac{\alpha}{2}}}{r_N q^j} \asymp r_N^{-\alpha/2}.
\end{align*}
For $s_j\equiv s+2\log j$, we have
\begin{align*}
\tau_{N,q}&\lesssim \max_{1\leq j\leq l} \left(\sqrt{s+2\log j}+ \frac{s+2\log j}{\sqrt{N}r_N q^j}\right)\\
&\lesssim   \sqrt{s\vee \log \log(1/r_N)}+(s\vee 1)N^{-\frac{\alpha}{2(\alpha+2)}},
\end{align*}
and the probability estimate $
e\cdot \sum_{j=1}^l \exp(-s_j)=e\cdot e^{-s} \sum_{j=1}^l j^{-2}\leq Ce^{-s}$.
This proves that
\beqa\label{ineq:ratio_set_sd_1}
&\Prob\bigg(\sup_{f \in \mathcal{F}: r_N^2\leq \sigma_P^2 f\leq 1} \frac{\abs{\tilde{\G}_N^\pi (f)}}{\sigma_P f}\\
&\qquad\qquad  \gtrsim \left(\beta_{N,q}+\sqrt{s\vee \log \log(1/r_N)}+(s\vee 1)N^{-\frac{\alpha}{2(\alpha+2)}}\right)\bigg)\leq e^{-s}.\nonumber
\eeqa
Take $s=3\log N$ to conclude the first claim by Proposition \ref{prop:deviation_normalized_process} and handle the residue part in (\ref{ineq:decomposition_Gtilde}) by (A2-CLT).

For the second claim, note that the proofs of Proposition \ref{prop:deviation_normalized_process} go through by replacing $\sigma_P f$ with a larger term $\sqrt{Pf}$ (since $Pf^2\leq Pf$). Take $\phi(x)=x^2$. For notational convenience we take $r_N>0$ such that $r_N\cdot N^{1/(\alpha+2)}\to \infty$ from now on. Note that
\begin{align*}
\beta_{N,q}\lesssim \max_{1\leq j\leq l} \frac{(r_N q^j)^{1-\frac{\alpha}{2}}}{r_N^2 q^{2j}}\asymp r_N^{-(1+\frac{\alpha}{2})},
\end{align*}
where in the first inequality we note that we only need an upper bound on $Pf^2$, and hence an upper bound for $Pf$ suffices. For $s_j\equiv s+2\log j$, we have
\begin{align*}
\tau_{N,q}&\lesssim \max_{1\leq j\leq l} \bigg((r_Nq^j)^{-1}\sqrt{s+2\log j}+ \frac{s+2\log j}{\sqrt{N}r_N^2 q^{2j}}\bigg)\\
&\lesssim   \sqrt{r_N^{-2}\big(s\vee \log \log(1/r_N)\big) }+(s\vee 1)(\sqrt{N}r_N^2)^{-1}.
\end{align*}
This proves that with $\bar{\gamma}_N\equiv N^{-1/2} r_N^{-(1+\frac{\alpha}{2})}\to 0$,
\begin{align*}
&\Prob\bigg(\sup_{f \in \mathcal{F}: Pf\geq r_N^2} \frac{\abs{\Prob_N^\pi f-Pf\cdot (N^{-1}\sum_{i=1}^N \xi_i/\pi_i)}}{Pf}\\
&\qquad\qquad \gtrsim \bar{\gamma}_N+  \sqrt{(Nr_N^2)^{-1}\big(s\vee \log \log(1/r_N) \big)}+(s\vee 1)(Nr_N^2)^{-1}\bigg)\\
&\leq C e^{-s}.
\end{align*}
By taking again $s=3 \log N$, it follows that
\begin{align*}
\sup_{f \in \mathcal{F}: Pf\geq r_N^2} \frac{\abs{\Prob_N^\pi f-Pf\cdot (N^{-1}\sum_{i=1}^N \xi_i/\pi_i)}}{Pf} = \mathfrak{o}_{\mathbf{P}}(1).
\end{align*}
The claim now follows by noting that $N^{-1}\sum_{i=1}^N \xi_i/\pi_i\to 1$ in probability by (A2-CLT) (actually here (A2-LLN) suffices).
\end{proof}

\begin{proof}[Proof of Theorem \ref{thm:weighted_CLT}]
	We only need to prove asymptotic equi-continuity of the weighted Horvitz-Thompson empirical process. More specifically, we only need to establish that for any $\epsilon>0$ and any $\delta_N\to 0$,
	\begin{align*}
	\lim_{N \to \infty}\Prob\bigg(\sup_{f \in \mathcal{F}: r_N<\sigma_P f\leq \delta_N} \frac{\G_N^\pi (f)}{\phi(\sigma_P f)}>\epsilon\bigg)=0.
	\end{align*}
	By similar calculations as in Theorem \ref{thm:local_asymp_moduli}, this time setting $s_j = 3\log \log (1/(r_N2^j))$, we have for $N$ large enough (and hence $\delta_N>0$ small enough), 
	\begin{align*}
	\beta_{N,2}&\lesssim \max_{1\leq j\leq l} \frac{ (r_N 2^j)^{1-\frac{\alpha}{2}} }{\phi(r_N2^j)}\leq \epsilon/2,\\
	\tau_{N,2}&\asymp \max_{1\leq j\leq l} \frac{r_N 2^j \sqrt{\log \log (1/(r_N2^j))} + \log \log (1/(r_N2^j))/\sqrt{N}}{\phi(r_N 2^j)}\\
	& \leq \epsilon\cdot \max_{1\leq j\leq l} \frac{r_N 2^j \sqrt{\log \log (1/(r_N2^j))} + \log \log (1/(r_N2^j))/\sqrt{N}}{(r_N 2^j)^{1-\frac{\alpha}{2}} \sqrt{\log \log (1/(r_N2^j))}}\\
	&\lesssim \epsilon \bigg( \delta_N^{\alpha/2}+\frac{\sqrt{\log \log N}}{\sqrt{N} r_N^{1-\frac{\alpha}{2}}}\bigg)= \mathfrak{o}(\epsilon).
	\end{align*}
	The probability estimate is
	\begin{align*}
	e\cdot \sum_{j=1}^l \exp(-s_j)&=e\cdot \sum_{j=1}^l \frac{1}{\log^3(1/(r_N2^j))}\\
	&\lesssim \sum_{j=1}^l \frac{1}{\log^3(1/(r_N2^j))} \big(\log(1/(r_N2^j))-\log(1/(r_N2^{j+1}))\big)\\
	&\leq \int_{\log(1/\delta_N)}^{\log(1/r_N)} x^{-3}\ \d{x}\to 0,
	\end{align*}
	as $N\to \infty$. Now apply Proposition \ref{prop:deviation_normalized_process} we see that
	\begin{align*}
	\lim_{N \to \infty}\Prob\bigg(\sup_{f \in \mathcal{F}: r_N<\sigma_P f\leq \delta_N} \frac{\tilde{\G}_N^\pi (f)}{\phi(\sigma_P f)}>\epsilon\bigg)=0.
	\end{align*}
	The residue part in (\ref{ineq:decomposition_Gtilde}) can be handled using (A2-CLT) as $\delta_N \to 0$. 
\end{proof}

\begin{proof}[Proof of Theorem \ref{thm:calibrate_HT}]
	For (1), note that
	\begin{align}\label{ineq:decom_calibrate_1}
	\abs{(\Prob_N^{\pi,c} -P)(f)}&\leq \abs{(\Prob_N^{\pi} -P)(f)}+\biggabs{\frac{1}{N}\sum_{i=1}^N \frac{\xi_i}{\pi_i}f(Y_i)\cdot \big(G(Z_i^\top \hat{\alpha}_N)-1\big)}\\
	&\lesssim \abs{(\Prob_N^{\pi} -P)(f)}+ \Prob_N \abs{f}\cdot \sup_{z \in \mathcal{Z}}\abs{G(z^\top \hat{\alpha}_n)-1}=\mathfrak{o}_{\mathbf{P}}(1)\nonumber
	\end{align}
	where the convergence follows from Theorem \ref{thm:ULLN_HT} and the assumptions on $G$.
	
	For (2), similar to (1), we have
	\begin{align}\label{ineq:decom_calibrate_2}
	\sup_{f \in \mathcal{F}_{\delta_N}}\abs{\G_N^{\pi,c}(f) }
	&\lesssim \sup_{f \in \mathcal{F}_{\delta_N}}\abs{\G_N^{\pi}(f) }+ \sup_{f \in \mathcal{F}_{\delta_N}} \sqrt{\Prob_N f^2}\cdot \sqrt{N}\sup_{z \in \mathcal{Z}}\abs{G(z^\top \hat{\alpha}_n)-1}.
	\end{align}
%	Here $\eta_i^G\equiv \frac{\xi_i}{\pi_i}\big(G(Z_i^\top \hat{\alpha}_N)-1\big)$ are independent from $Y_i$'s conditionally on $Z^{(N)}$. It is easy to see that the first and third terms in (\ref{ineq:decom_calibrate_2}) are $\mathfrak{o}_{\mathbf{P}}(1)$. To handle the second term of (\ref{ineq:decom_calibrate_2}), note that $\max_{1\leq i\leq N} \abs{\eta_i^G}= \mathcal{O}_{\mathbf{P}}(1)$. Hence for any fixed $\epsilon>0$, there exists $N_\epsilon \in \N$ such that for $N\geq N_\epsilon$, the event $\mathcal{E}_\epsilon\equiv \{\max_{1\leq i\leq N}\abs{\eta_i^G}\leq 1\}$ occurs with probability at least $1-\epsilon$
	Note that $\mathcal{F}$ is $P$-Donsker implies that $\mathcal{F}$ is $P$-Glivenko-Cantelli. Now using the characterization for Glivenko-Cantelli classes (cf. Theorem 3.7.14 (a) and (c) in \cite{gine2015mathematical}), we can conclude that $\mathcal{F}^2$ is also $P$-Glivenko-Cantelli (since we assumed that $PF^2<\infty$). This implies that $\sup_{f \in \mathcal{F}_{\delta_N}} \sqrt{\Prob_N f^2}=\mathfrak{o}_{\mathbf{P}}(1)$. Take limit as $N\to \infty$ in the above display, we see that $\sup_{f \in \mathcal{F}_{\delta_N}}\abs{\G_N^{\pi,c}(f) }= \mathfrak{o}_{\mathbf{P}}(1)$ by the assumption on $\hat{\alpha}_N$. This proves the asymptotic equi-continuity for $\G_N^{\pi,c}$. 
	
	To prove (3), we only need to use the decompositions (\ref{ineq:decom_calibrate_1})-(\ref{ineq:decom_calibrate_2}) and the corresponding theorems. 
\end{proof}

\begin{proof}[Proof of Proposition \ref{prop:cov_structure_cal}]
	We first prove (\ref{ineq:cov_structure_cal_0}). Let $\psi_{\alpha}\equiv \frac{\xi}{\pi} G(Z^\top \alpha) Z-Z$, then $\Prob_N \psi_{\hat{\alpha}_N} = 0$. Hence in the notation of Theorem \ref{thm:Z_est} (with the usual empirical process), we have $\Psi_N(\alpha)= \Prob_N \frac{\xi}{\pi} G(Z^\top \alpha) Z-\Prob_N Z$, and $\Psi(\alpha) = P[(G(Z^\top \alpha)-1)Z]$. The Fr\'echet derivative of $\Psi$ at $\alpha=0$ is given by
	\begin{align*}
	\dot{\Psi}(0)=\frac{\d{}}{\d{\alpha}}P(G(Z^\top \alpha)-1)Z\bigg\lvert_{\alpha=0}=G'(0) \cdot P (ZZ^\top).
	\end{align*}
	Since $\{Z^\top \alpha: \alpha \in \mathcal{A}_c\}$ is $P$-Glivenko-Cantelli, by Theorem 3 of \cite{van2000preservation}, $\{G(Z^\top \alpha): \alpha \in \mathcal{A}_c\}$ is $P$-Glivenko-Cantelli. Since $Z$ is bounded, it is easy to see by Proposition 2 of \cite{van2000preservation} (which is due to \cite{gine1984some}) that $\{G(Z^\top \alpha)Z: \alpha \in \mathcal{A}_c\}$ is also $P$-Glivenko-Cantelli. Hence
	\begin{align*}
	\abs{\Psi(\hat{\alpha}_N)-\Psi(0)}&= \abs{\Psi(\hat{\alpha}_N)-\Psi_N(\hat{\alpha}_N)}\\
	&\leq \sup_{\alpha \in \mathcal{A}_c } \abs{(\Prob_N^\pi-P)(G(Z^\top \alpha) Z)}+\abs{(\Prob_N-P)Z}=\mathfrak{o}_{\mathbf{P}}(1).
	\end{align*}
	This means that $\hat{\alpha}_N=\mathfrak{o}_{\mathbf{P}}(1)$. Furthermore, for some $\tilde{\alpha}$ such that $\pnorm{\tilde{\alpha}}{}\leq \pnorm{\hat{\alpha}_N}{}$,
	\begin{align*}
	\pnorm{\G_N(\psi_{\hat{\alpha}_N}-\psi_0)}{}& = \biggpnorm{\G_N\bigg(\frac{\xi}{\pi}\big(G(Z^\top \hat{\alpha}_N)-1\big) Z\bigg)}{}\\
	&=  \pnorm{(\Prob_N^\pi -P) ZZ^\top G'(Z^\top \tilde{\alpha}) \cdot \sqrt{N}\hat{\alpha}_N}{}.
	%&\lesssim \mathfrak{o}_{\mathbf{P}}(\sqrt{N}\pnorm{\hat{\alpha}_N}{})
	\end{align*}
	Since the class $\{ZZ^\top G'(Z^\top \alpha): \pnorm{\alpha}{}\leq \delta\}$ is $P$-Glivenko-Cantelli for small enough $\delta>0$ by similar arguments as above, consistency of $\hat{\alpha}_N$ and Theorem \ref{thm:ULLN_HT} yield that 
	\begin{align*}
	\pnorm{\G_N(\psi_{\hat{\alpha}_N}-\psi_0)}{}= \mathfrak{o}_{\mathbf{P}}(\sqrt{N}\pnorm{\hat{\alpha}_N}{}).
	\end{align*}
	Now it follows from Theorem \ref{thm:Z_est} that
	\begin{align*}
	\sqrt{N}\hat{\alpha}_N=-(G'(0))^{-1} (P(ZZ^\top))^{-1} \G_N \psi_0+\mathfrak{o}_{\mathbf{P}}(1),
	\end{align*}
	and the claim (\ref{ineq:cov_structure_cal_0}) follows by noting that $\G_N\psi_0 = (\G_N^\pi-\G_N)Z$. Now we have
	\begin{align*}
	\G_N^{\pi,c}(f)& = \frac{1}{\sqrt{N}}\sum_{i=1}^N \frac{\xi_i}{\pi_i} \big(G(Z_i^\top \hat{\alpha}_N)-1\big)f(Y_i) + \G_N^\pi (f)\\
	& = \frac{1}{N}\sum_{i=1}^N \frac{\xi_i}{\pi_i}f(Y_i) G'(Z_i^\top \tilde{\alpha}) Z_i^\top \big(\sqrt{N}\hat{\alpha}_N\big) +\G_N^\pi (f)\\
	& = G'(0)\cdot\frac{1}{N}\sum_{i=1}^N \frac{\xi_i}{\pi_i}f(Y_i)  Z_i^\top \big(\sqrt{N}\hat{\alpha}_N\big) +\G_N^\pi (f)+\mathfrak{o}_{\mathbf{P}}(1)\\
	& = G'(0) P(f(Y)Z^\top)\big(\sqrt{N}\hat{\alpha}_N\big) +\G_N^\pi (f)+\mathfrak{o}_{\mathbf{P}}(1)\\
	& = -P(f(Y)Z^\top) (P(ZZ^\top))^{-1} (\G_N^\pi-\G_N)Z+\G_N^\pi(f)+\mathfrak{o}_{\mathbf{P}}(1)\\
	&= -(\G_N^\pi-\G_N)\tilde{Z}+(\G_N^\pi-\G_N)(f)+ \G_N(f)+\mathfrak{o}_{\mathbf{P}}(1)\\
	& = (\G_N^\pi-\G_N)(g_f) +\G_N(f)+\mathfrak{o}_{\mathbf{P}}(1)
	\end{align*}
	where $\tilde{Z}\equiv P(f(Y)Z^\top) (P(ZZ^\top))^{-1} Z \in \R$ and $g(Y,Z)\equiv g_f(Y,Z)\equiv f-\tilde{Z}=f(Y)-P(f(Y)Z^\top) (P(ZZ^\top))^{-1} Z$ are bounded random variables by the assumptions. From here we use the same strategy as in the proof of Proposition \ref{prop:cov_G_pi}: for any $\bm{g}=(g_\ell)_{\ell=1}^k $ and $\bm{a}=(a_\ell)_{\ell=1}^k$, let $g\equiv \sum_{\ell=1}^k a_\ell g_\ell= \bm{a}^\top \bm{g}$. Then
	\begin{align*}
	\G_N^{\pi,c}(f)& = \sqrt{N{S}_N^2}\cdot \frac{1}{{S}_N}\bigg(\frac{1}{N}\sum_{i=1}^N\frac{\xi_i}{\pi_i}g(Y_i,Z_i)-\frac{1}{N}\sum_{i=1}^N g(Y_i,Z_i)\bigg)+\G_N(f)+\mathfrak{o}_{\mathbf{P}}(1)
	\end{align*}
	where $
	{S}_N^2=\frac{1}{N^2}\sum_{1\leq i,j\leq N}\frac{\pi_{ij}-\pi_i\pi_j}{\pi_i\pi_j}g(Y_i,Z_i)g(Y_j,Z_j)$ satisfies
	\begin{align*}
	N{S}_N^2&\to_{\Prob_{(Y,Z)}}  \bm{a}^\top \big(\mu_{\pi1}P[\bm{g}\bm{g}^\top]+\mu_{\pi2}(P\bm{g})(P\bm{g}^\top\big)\big)\bm{a}.
	\end{align*}
	On the other hand, the asymptotic variance of $\G_N(f)$ is given by
	\begin{align*}
	\bm{a}^\top\big(P[\bm{f}\bm{f}^\top]- (P\bm{f})(P\bm{f}^\top)\big)\bm{a}.
	\end{align*}
	The claim of the proposition now follows from Lemma \ref{lem:sum_cond_CLT}.
\end{proof}

\begin{proof}[Proof of Corollary \ref{cor:P_N_center_Donsker}]
	Asymptotic equicontinuity follows from the decomposition $\sqrt{n}(\Prob_N^\pi - \Prob_N) = \sqrt{n}(\Prob_N^\pi - P)-\sqrt{n}(\Prob_N-P)$. The covariance structure can be checked similarly to the proof of Proposition \ref{prop:cov_G_pi} so we omit the details.
\end{proof}

\begin{proof}[Proof of Corollary \ref{cor:Hajek_Donsker}]
	Note that
	\begin{align*}
	\sqrt{n}\big(\Prob_N^{\pi, H}-\Prob_N\big) = \sqrt{n/N} \big(\mathbb{Y}_N + \big(N/\hat{N}-1\big) \tilde{\G}_N^\pi \big),
	\end{align*}
	where
	\begin{align*}
	\mathbb{Y}_N(f)&\equiv \frac{1}{\sqrt{N}}\sum_{i=1}^{N}\bigg(\frac{\xi_i}{\pi_i}-1\bigg) \big(f(Y_i)-Pf\big),\\
	\tilde{\G}_N^\pi(f)& \equiv \frac{1}{\sqrt{N}}\sum_{i=1}^N \frac{\xi_i}{\pi_i}\big(f(Y_i)-Pf\big) = \G_N^\pi f + Pf\cdot \frac{1}{\sqrt{N}} \sum_{i=1}^N \bigg(1-\frac{\xi_i}{\pi_i}\bigg).
	\end{align*}
	Since $N/\hat{N}-1 = \mathfrak{o}_{\mathbf{P}}(1)$ by (A1), and $\sup_{f \in \mathcal{F}}\abs{\tilde{\G}_N^\pi(f)} = \mathcal{O}_{\mathbf{P}}(1)$, it follows that the limit behavior of $\sqrt{n}\big(\Prob_N^{\pi, H}-\Prob_N\big) $ is determined by $\mathbb{Y}_N$. The covariance structure can be verified along the lines of the proof of Proposition \ref{prop:cov_G_pi} (and is actually easier) so we omit the details.
\end{proof}

\begin{lemma}\label{lem:transfer_probability}
	$\Delta_N = \mathfrak{o}_{\mathbf{P}}(1)$ if and only if $\Delta_N\equiv \mathfrak{o}_{\Prob_d}(1)$ in $\Prob_{(Y,Z)}$-probability.
%	The following statements are valid.
%	\begin{enumerate}
%		\item $\Delta_N = \mathcal{O}_{\mathbf{P}}(1)$ if and only if $\Delta_N\equiv \mathcal{O}_{\Prob_d}(1)$ in $\Prob_{(Y,Z)}$-probability.
%		\item $\Delta_N = \mathfrak{o}_{\mathbf{P}}(1)$ if and only if $\Delta_N\equiv \mathfrak{o}_{\Prob_d}(1)$ in $\Prob_{(Y,Z)}$-probability.
%	\end{enumerate}
\end{lemma}
\begin{proof}
	Suppose $\Delta_N = \mathfrak{o}_{\mathbf{P}}(1)$. Then for any $\epsilon,\delta>0$,
	\begin{align*}
	\Prob_{(Y,Z)} \big(\Prob_{d|(Y,Z)}(\abs{\Delta_N}>\epsilon)>\delta\big)&\leq \delta^{-1} \E_{(Y,Z)} \Prob_{d|(Y,Z)}(\abs{\Delta_N}>\epsilon) \\
	&= \delta^{-1} \Prob\big(\abs{\Delta_N}>\epsilon\big)\to 0.
	\end{align*}
	Conversely, suppose $\Delta_N\equiv \mathfrak{o}_{\Prob_d}(1)$ in $\Prob_{(Y,Z)}$-probability. For any $\epsilon,\delta>0$,
	\begin{align*}
	\Prob\big(\abs{\Delta_N}>\epsilon\big)&=\E_{(Y,Z)} \Prob_{d|(Y,Z)}(\abs{\Delta_N}>\epsilon)\big(\bm{1}_{\Prob_{d|(Y,Z)}(\abs{\Delta_N}>\epsilon)>\delta }+ \bm{1}_{\Prob_{d|(Y,Z)}(\abs{\Delta_N}>\epsilon)\leq \delta }\big)\\
	&\leq \E_{(Y,Z)} \bm{1}_{\Prob_{d|(Y,Z)}(\abs{\Delta_N}>\epsilon)>\delta } + \delta\\
	& = \Prob_{(Y,Z)} \big(\Prob_{d|(Y,Z)}(\abs{\Delta_N}>\epsilon)>\delta\big) + \delta \to 0
	\end{align*}
	as $N \to \infty$ followed by $\delta \to 0$.
\end{proof}

\begin{proof}[Proof of Corollaries \ref{cor:conditional_ULLN} and \ref{cor:conditional_UCLT}]
	The claim of Corollary \ref{cor:conditional_ULLN} follows from Lemma \ref{lem:transfer_probability}. For Corollary \ref{cor:conditional_UCLT}, similar to the proof of Theorem 2.2 of \cite{praestgaard1993exchangeably}, we only need to verify that with $\bar{\G}_N^\pi = \sqrt{n}(\Prob_N^\pi - \Prob_N)$,
	\begin{enumerate}
		\item $(\bar{\G}_N^\pi(f_1),\ldots, \bar{\G}_N^\pi(f_\ell))\rightsquigarrow (\bar{\G}^\pi(f_1),\ldots,\bar{\G}^\pi(f_\ell))$ for any $f_1,\ldots,f_\ell \in \mathcal{F}$ $\Prob_{(Y,Z)}$-a.s., and
		\item for every $\epsilon>0$ and $\delta_N \to 0$, it holds that 
		\begin{align*}
		\lim_{N \to \infty}\Prob_{(Y,Z)} \bigg(\E_{d|(Y,Z)} \sup_{f \in \mathcal{F}_{\delta_N}} \abs{\bar{\G}_N^\pi(f)}>\epsilon\bigg)=0.
		\end{align*}
	\end{enumerate}
	(1) can be checked using Cram\'er-Wold device as in the proof of Proposition \ref{prop:cov_G_pi} along with the countability of $\mathcal{F}$. For (2), it suffices to check $\E\sup_{f \in \mathcal{F}_{\delta_N}} \abs{\bar{\G}_N^\pi(f)}\to 0$. This is a direct consequence of the proof of Theorem \ref{thm:weak_convergence_HT}.
\end{proof}

\section{Proofs for Section \ref{section:applications}}\label{section:proof_app}

In this section we collect proofs for the results in Section \ref{section:applications}.

\begin{proof}[Proof of Theorem \ref{thm:M_estimation}]
	It suffices to prove
	\begin{align}\label{ineq:M_est_0}
	&\Prob\bigg(\sup_{f \in \mathcal{F}: \mathcal{E}_P(f)\geq r_N^2} \biggabs{\frac{\mathcal{E}_{\Prob_N^\pi}(f)}{\mathcal{E}_P(f)}-1}\geq 3/4\bigg)\\
	&\qquad\leq \frac{C_3}{s}e^{-s/C_3}+\Prob\bigg(\biggabs{\frac{1}{\sqrt{N}}\sum_{i=1}^N\bigg(\frac{\xi_i}{\pi_i}-1\bigg) }>t\bigg).\nonumber
	\end{align}
	We remind the readers that the constants $C_i$'s below may not agree with that in the statement of the theorem. Let $\mathcal{F}_j\equiv \{f-g: f,g \in \mathcal{F}_{\mathcal{E}}(r_N2^{j})\}$ for $1\leq j\leq l$ where $l$ is the smallest integer such that $r_N^2 2^{2l}\geq \sup_{f \in \mathcal{F}} Pf - \inf_{f \in \mathcal{F}}Pf$. By Proposition \ref{prop:Talagrand_HT}, there exists some $C_1>1$ only depending on $\pi_0$ such that for any $s_j\equiv s 2^{2j}$ with $s\geq 0$,
	\begin{align*}
	\Prob\bigg[\pnorm{\tilde{\G}_N^\pi}{\mathcal{F}_j}\geq C_1\bigg(\E \pnorm{\G_N}{\mathcal{F}_j} + \sqrt{\sigma^2_j s_j}+\frac{s_j}{\sqrt{N}}\bigg)\bigg]\leq e\cdot \exp\big(-s_j\big)
	\end{align*}
	where $\sigma_j^2 = \sup_{f \in \mathcal{F}_j} \sigma_P^2 f\lesssim (r_N 2^j)^{2/\kappa}$. Hence by a union bound (and boosting $C_1$ if necessary),
	\begin{align}\label{ineq:M_est_1}
	&\Prob\bigg[\max_{1\leq j\leq l} \frac{ \sup_{f \in \mathcal{F}_j}\abs{N^{-1/2}\tilde{\G}_N^\pi(f)} }{r_N^2 2^{2j}}\geq \frac{1}{16}+C_1 \bigg(\sqrt{\frac{s}{N r_N^{4-\frac{2}{\kappa}}}}+\frac{s}{Nr_N^2}\bigg) \bigg]\leq \frac{C_1}{s}e^{-s/C_1},
	\end{align}
	where in the above inequality we used Lemma \ref{lem:local_maximal_ineq} to deduce that
	\begin{align*}
	C_1\max_{1\leq j\leq l} \frac{\E \pnorm{\G_N}{\mathcal{F}_j}}{\sqrt{N}r_N^2 2^{2j}}\leq C_2 \cdot \frac{(r_N 2^j)^{\frac{1}{\kappa}(1-\frac{\alpha}{2})}}{\sqrt{N}(r_N 2^j)^2}\bigg(1+\frac{(r_N 2^j)^{\frac{1}{\kappa}(1-\frac{\alpha}{2})}}{\sqrt{N}(r_N 2^j)^{2/\kappa}}\bigg)\leq 1/16,
	\end{align*}
	as long as $r_N$ is chosen so that 
	\begin{align*}
	r_N\geq (32C_2)^{\frac{\kappa}{2\kappa-1+\alpha/2}} N^{-\frac{\kappa}{4\kappa-2+\alpha}}.
	\end{align*}
    Let the event in (\ref{ineq:M_est_1}) denote $E_s$, and let $F_t\equiv \{\bigabs{\frac{1}{\sqrt{N}} \sum_{i=1}^N\big(\frac{\xi_i}{\pi_i}-1\big)}\leq t \}$. Write $f_0\equiv \arg \min_{f \in \mathcal{F}} Pf$. On the event $E_s\cap F_t$, we have for any $f \in \mathcal{F}_{\mathcal{E}}(r_N2^j)\setminus \mathcal{F}_{\mathcal{E}}(r_N2^{j-1})$ and $f' \in \mathcal{F}_{\mathcal{E}}(\sigma)$ for some $0<\sigma<r_N 2^j$,
	\begin{align*}
	\mathcal{E}_P(f) &= P (f-f')+\big[Pf'-Pf_{0}\big]\leq P (f-f')+\sigma\\
	&\leq \Prob_N^\pi (f-f')+\sigma+ \pnorm{\Prob_N^\pi-P}{\mathcal{F}_j}\\
	&\leq \Prob_N^\pi (f-f')+\sigma+ \pnorm{N^{-1/2}\tilde{\G}_N^\pi}{\mathcal{F}_j}+N^{-1/2}\pnorm{Pf}{\mathcal{F}_j} \cdot t\\
	&\leq \mathcal{E}_{\Prob_N^\pi}(f)+\sigma+ \bigg[\frac{1}{16}+C_1 \bigg(\sqrt{\frac{s}{N r_N^{4-\frac{2}{\kappa}}}}+\frac{s}{Nr_N^2}\bigg)\bigg] r_N^2 2^{2j}+ N^{-1/2} L (r_N 2^j)^{1/\kappa} t \\
	& \leq \mathcal{E}_{\Prob_N^\pi}(f)+\sigma+ \bigg[\frac{1}{8}+C_1 \bigg(\sqrt{\frac{s}{N r_N^{4-\frac{2}{\kappa}}}}+\frac{s}{Nr_N^2}\bigg)\bigg] 4\mathcal{E}_P(f),
	\end{align*}
	provided 
	\begin{align*}
	r_N\geq \bigg(\frac{256 L^2 t^2}{N}\bigg)^{\frac{\kappa}{4\kappa-2}}.
	\end{align*}
	Since $\sigma>0$ is taken arbitrarily, we see that on the event $E_s\cap F_t$, it holds that
	\begin{align*}
	\frac{\mathcal{E}_{\Prob_N^\pi}(f)}{\mathcal{E}_P(f)}\geq 1-\bigg(\frac{1}{2}+4C_1\sqrt{\frac{s}{Nr_N^{4-\frac{2}{\kappa}}} } +4C_1\frac{s}{Nr_N^2}\bigg)
	\end{align*}
	for all $f \in \mathcal{F}$ such that $\mathcal{E}_P(f)\geq r_N^2$. Further choosing 
	\begin{align*}
	r_N\geq  (32C_1)^{\frac{2\kappa}{4\kappa-2}} \bigg(\frac{s}{N}\bigg)^{\frac{\kappa}{4\kappa-2} }\vee (32 C_1)^{1/2}\bigg(\frac{s}{N}\bigg)^{1/2}\equiv C_3 \bigg(\frac{s}{N}\bigg)^{\frac{\kappa}{4\kappa-2} },
	\end{align*}
	we have that $\mathcal{E}_{P} (\hat{f}_N^\pi)<r_N^2$. Hence for any $f \in \mathcal{F}_{\mathcal{E}}(r_N 2^j)\setminus \mathcal{F}_{\mathcal{E}}(r_N 2^{j-1})$,  
	\begin{align*}
	\mathcal{E}_{\Prob_N^\pi}(f)& = \Prob_N^\pi f-\Prob_N^\pi \hat{f}_N^\pi \leq Pf - P\hat{f}_N^\pi+\pnorm{\Prob_N^\pi-P}{\mathcal{F}_j}\\
	& \leq \mathcal{E}_{P}(f)+ \bigg[\frac{1}{8}+C_1 \bigg(\sqrt{\frac{s}{N r_N^{4-\frac{2}{\kappa}}}}+\frac{s}{Nr_N^2}\bigg)\bigg] 4\mathcal{E}_P(f).
	\end{align*}
	This entails
	\begin{align*}
	\frac{\mathcal{E}_{\Prob_N^\pi}(f)}{\mathcal{E}_P(f)}\leq 1+\bigg(\frac{1}{2}+4C_1\sqrt{\frac{s}{Nr_N^{4-\frac{2}{\kappa}}} } +4C_1\frac{s}{Nr_N^2}\bigg)
	\end{align*}
	for all $f \in \mathcal{F}$ such that $\mathcal{E}_P(f)\geq r_N^2$. The claim (\ref{ineq:M_est_0}) follows by noting that the choice of $r_N$ entails that the term in the parenthesis above is no larger than $3/4$.
\end{proof}

\begin{proof}[Proof of Theorem \ref{thm:Z_est}]
	The proof adapts that of Theorem 3.3.1 of \cite{van1996weak}. By definition of $\hat{\theta}_N^\pi$, we have
	\begin{align}\label{ineq:Z_est_1}
	\sqrt{N} \big(\Psi(\hat{\theta}_N^\pi)-\Psi(\theta_0)\big)\equiv \sqrt{N} \big(\Psi(\hat{\theta}_N^\pi)-\Psi_N(\hat{\theta}_N^\pi)\big)= - \sqrt{N}(\Psi_N-\Psi)(\theta_0) +R_N,
	\end{align}
	where
	\begin{align}\label{ineq:Z_est_2}
	\pnorm{R_N}{\mathcal{H}}&= \pnorm{\G_N^\pi(\psi_{\hat{\theta}_N^\pi,h}-\psi_{\theta_0,h})}{\mathcal{H}}\\
	&\lesssim \pnorm{\G_N(\psi_{\hat{\theta}_N^\pi,h}-\psi_{\theta_0,h})}{\mathcal{H}}+\pnorm{P(\psi_{\hat{\theta}_N^\pi,h}-\psi_{\theta_0,h})}{\mathcal{H}}\biggabs{\frac{1}{\sqrt{N}}\sum_{i=1}^N \bigg(\frac{\xi_i}{\pi_i}-1\bigg) } \nonumber\\
	% \biggabs{\frac{1}{\sqrt{N}}\sum_{i=1}^N \bigg(\frac{\xi_i}{\pi_i}-1\bigg) }\\
	&\leq (1+\mathcal{O}_{\mathbf{P}}(N^{-1/2})) \pnorm{\G_N(\psi_{\hat{\theta}_N^\pi,h}-\psi_{\theta_0,h})}{\mathcal{H}} + \pnorm{\Prob_N^\pi \psi_{\theta_0,h}}{\mathcal{H}}\cdot \mathcal{O}_{\mathbf{P}}(1)  \nonumber\\
	&\leq (1+\mathcal{O}_{\mathbf{P}}(N^{-1/2})) \pnorm{\G_N(\psi_{\hat{\theta}_N^\pi,h}-\psi_{\theta_0,h})}{\mathcal{H}} + \pnorm{(\Prob_N-P) \psi_{\theta_0,h}}{\mathcal{H}}\cdot \mathcal{O}_{\mathbf{P}}(1)  \nonumber\\
	&=\mathfrak{o}_{\mathbf{P}}\big(1+\sqrt{N}\pnorm{\hat{\theta}_N^\pi-\theta_0}{}\big).\nonumber
	\end{align}
	By Fr\'echet differentiability of $\Psi$ at $\theta_0$ and continuous invertibility of $\dot{\Psi}_{\theta_0}$, we have for all $\theta$ close enough to $\theta_0$,
	\begin{align}\label{ineq:Z_est_3}
	\pnorm{\Psi({\theta})-\Psi(\theta_0)}{\mathcal{H}}\gtrsim \pnorm{\theta-\theta_0}{}+\mathfrak{o}\big(\pnorm{\theta-\theta_0}{}\big).
	\end{align}
	Combining (\ref{ineq:Z_est_1})-(\ref{ineq:Z_est_3}) we obtain
	\begin{align*}
	\sqrt{N} \pnorm{\hat{\theta}_N^\pi-\theta_0}{}(1+\mathfrak{o}_{\mathbf{P}}(1))\lesssim \mathcal{O}_{\mathbf{P}}(1)+ \mathfrak{o}_{\mathbf{P}}\big(1+\sqrt{N}\pnorm{\hat{\theta}_N^\pi-\theta_0}{}\big),
	\end{align*}
	from which we conclude that $\sqrt{N}\pnorm{\hat{\theta}_N^\pi-\theta_0}{} = \mathcal{O}_{\mathbf{P}}(1)$, and hence $\pnorm{R_N}{\mathcal{H}}=\mathfrak{o}_{\mathbf{P}}(1)$. The claim now follows by the continuous invertibility of $\dot{\Psi}_{\theta_0}$.
	%
	%For calibrated $\hat{\theta}_N^{\pi,G}$, we only need to use the decomposition (note the difference from (\ref{ineq:decom_calibrate_1}))
	%\begin{align*}
	%\abs{(\Prob_N^{\pi,G} -P)(f)}&\leq \abs{(\Prob_N^{\pi} -P)(f)}+\biggabs{\frac{1}{N}\sum_{i=1}^N \frac{\xi_i}{\pi_i}f(X_i)\cdot \big(G(Z_i^\top \hat{\alpha}_N)-1\big)}\\
	%&\leq \abs{(\Prob_N^{\pi} -P)(f)}+ \Prob_N \abs{f}\cdot \sup_{z \in \mathcal{Z}}\abs{G(z^\top \hat{\alpha}_n)-1}
	%\end{align*}
\end{proof}

\begin{proof}[Proof of Theorem \ref{thm:semipara}]
	First note that
	\begin{align*}
	&\sqrt{N} P_{\theta_0,\eta_0} \big(\tilde{m}(\hat{\theta}_N^\pi,\hat{\eta}_N^\pi)-\tilde{m}(\theta_0,\eta_0)\big)\\
	&=-\sqrt{N} (\Prob_N^\pi- P_{\theta_0,\eta_0})\big(\tilde{m}(\hat{\theta}_N^\pi,\hat{\eta}_N^\pi)-\tilde{m}(\theta_0,\eta_0)\big)+\sqrt{N} \Prob_N^\pi \big(\tilde{m}(\hat{\theta}_N^\pi,\hat{\eta}_N^\pi)-\tilde{m}(\theta_0,\eta_0)\big)\\
	&= -\G_N^\pi \big(\tilde{m}(\hat{\theta}_N^\pi,\hat{\eta}_N^\pi)-\tilde{m}(\theta_0,\eta_0)\big) -\sqrt{N}\Prob_N^\pi \tilde{m}(\theta_0,\eta_0)+\mathfrak{o}_{\mathbf{P}}(1)\\
	& = -\G_N^\pi \tilde{m}(\theta_0,\eta_0)+ \mathfrak{o}_{\mathbf{P}}(1).
	\end{align*}
	where in the second equality we used (\ref{def:par_semi}), and in the third equality we used (\ref{cond:no_bias}) and (S2), (S4). Now by (S3), the left hand side of the above display equals $-I_{\theta_0,\eta_0} \big(\sqrt{N}(\hat{\theta}_N^\pi-\theta_0)\big)+\mathfrak{o}_{\mathbf{P}}(1)$, and hence (S1) yields that
	\begin{align*}
	\sqrt{N}(\hat{\theta}_N^\pi-\theta_0) = I_{\theta_0,\eta_0}^{-1} \G_N^\pi \tilde{m}(\theta_0,\eta_0)+\mathfrak{o}_{\mathbf{P}}(1),
	\end{align*}
	as desired.
\end{proof}

\begin{proof}[Proof of Theorem \ref{thm:bayes}]
	We first verify the local Gaussianity condition of \cite{han2017bayes}. To this end, write $p_f^{(N)}\equiv \prod_{i=1}^N p_f(Y_i)^{\xi_i/\pi_i}$. Then for $\lambda\in \R$, by Proposition \ref{prop:multiplier_ineq},
	\begin{align*}
	&P_{f_0}^{(N)} \exp\bigg[\lambda \bigg(\log \frac{p_{f_0}^{(N)}}{p_{f_1}^{(N)}}-P_{f_0}^{(N)}\log \frac{p_{f_0}^{(N)}}{p_{f_1}^{(N)}}\bigg)\bigg]\\
	&\leq  P_{f_0}^{(N)}\exp\left[\biggabs{\lambda \sum_{i=1}^N \frac{\xi_i}{\pi_i}\left(\log \frac{p_{f_0}}{p_{f_1}}(Y_i)-P_{f_0}\log \frac{p_{f_0}}{p_{f_1}}\right) }\right]\\
	&= \sum_{p=1}^\infty (p!)^{-1} \abs{\lambda}^p \cdot P_{f_0}^{(N)}  \biggabs{ \sum_{i=1}^N \frac{\xi_i}{\pi_i}\left(\log \frac{p_{f_0}}{p_{f_1}}(Y_i)-P_{f_0}\log \frac{p_{f_0}}{p_{f_1}}\right) }^p\\
	&\leq  \sum_{p=1}^\infty (p!)^{-1} \abs{C_0\lambda}^p \cdot  P_{f_0}^{(N)}  \biggabs{ \sum_{i=1}^N \left(\log \frac{p_{f_0}}{p_{f_1}}(Y_i)-P_{f_0}\log \frac{p_{f_0}}{p_{f_1}}\right) }^p\\
	& = P_{f_0}^{(N)}\exp\left[\biggabs{C_0 \lambda \sum_{i=1}^N \left(\log \frac{p_{f_0}}{p_{f_1}}(Y_i)-P_{f_0}\log \frac{p_{f_0}}{p_{f_1}}\right) }\right]\\
	&\leq c_1' \exp\left[\psi_{\kappa_g' nd^2(f_0,f_1),\kappa_\Gamma' }(\lambda)\right],
	\end{align*}
	where in the last inequality we may adjust constants to handle the absolute value (cf. Theorem 2.3 of \cite{boucheron2013concentration}). 
	
	Now by (essentially) Lemma 1 of \cite{han2017bayes}, there exists a test $\phi_N\equiv \phi_N(D^{(N)})$ such that for any $j \in \N$,
	\begin{align*}
	P_{f_0}\phi_N \leq c_0e^{-N\delta_N^2/c_0},\quad \sup_{f\in \mathcal{F}: d(f,f_0)\geq j\delta_N}  P_{f}(1-\phi_N)\leq c_0 e^{-j^2N\delta_N^2/c_0}.
	\end{align*}
	By Lemma 12 of \cite{han2017bayes}, there exist constants $c_1,c_2>0$ such that $P_{f_0}(A_N^c)\leq c_2e^{-N\delta_N^2/c_2}$, where
	\begin{align*}
	A_N\equiv \bigg\{\int \frac{p_f^{(N)}}{p_{f_0}^{(N)}}\ \d{\Pi_N(f)}\geq \Pi_N\big(d(f,f_0)\leq \delta_N\big)e^{-N\delta_N^2/c_1}\bigg\}.
	\end{align*}
	Now for notational convenience, let $\hat{\Pi}_N^\pi\equiv \Pi_N^\pi(f\in \mathcal{F}: d(f,f_0)>\delta_N|D^{(N)})$, we have
	\begin{align*}
	P_{f_0} \hat{\Pi}_N^\pi & = P_{f_0} \hat{\Pi}_N^\pi \phi_N + P_{f_0} \hat{\Pi}_N^\pi (1-\phi_N)\bm{1}_{A_N^c}+P_{f_0} \hat{\Pi}_N^\pi (1-\phi_N)\bm{1}_{A_N}\\
	&\leq c_3 e^{-N\delta_N^2/c_3}+ P_{f_0} \hat{\Pi}_N^\pi (1-\phi_N)\bm{1}_{A_N}.
	\end{align*}
	On the other hand, let $\mathcal{F}_j\equiv \{f\in \mathcal{F}: j\delta_N<d(f,f_0)\leq (j+1) \delta_N\}$, we have by assumption
	\begin{align*}
	P_{f_0} \hat{\Pi}_N^\pi (1-\phi_N)\bm{1}_{A_N}&\lesssim \sum_{j=1}^\infty \frac{e^{-j^2 N\delta_N^2/c_0}\Pi_N(\mathcal{F}_j)}{\Pi_N\big(d(f,f_0)\leq \delta_N\big)e^{-N\delta_N^2/c_1}}\\
	&\lesssim \sum_{j=1}^\infty e^{-j^2 N\delta^2_N/c_4} \lesssim e^{- N\delta^2_N/c_5},
	\end{align*}
	as desired.
\end{proof}

Finally we prove Theorem \ref{thm:bvm}. First we need the following general result due to \cite{kleijn2012bernstein}.

\begin{proposition}\label{prop:bvm_generic}
	Suppose the following conditions hold:
	\begin{enumerate}
		\item (LAN condition) There exist random vectors $\Delta_{N,\theta_0}=\mathcal{O}_{\mathbf{P}}(1)$ and a non-singular matrix $I_{\theta_0}$ such that for every compact $K \subset \R^d$,
		\begin{align*}
		\sup_{h \in K} \biggabs{ N\Prob_N^\pi \log \frac{p_{\theta_0+h/\sqrt{N}} }{p_{\theta_0}}-h^\top I_{\theta_0} \Delta_{N,\theta_0}-\frac{1}{2}h^\top I_{\theta_0}h}=\mathfrak{o}_{\mathbf{P}}(1).
		\end{align*}
		\item (Sufficient mass condition) The prior $\Pi$ on $\Theta$ has a Lebesgue density being continuous and positive on a neighborhood of $\theta_0$.
		\item (Posterior contraction at $\sqrt{N}$-rate) For every $L_N\to \infty$,
		\begin{align*}
		P_{\theta_0} \Pi_N^\pi\big(\theta \in \Theta: \pnorm{\theta-\theta_0}{}>L_N/\sqrt{N}\big\lvert D^{(N)}\big)\to 0.
		\end{align*}
	\end{enumerate}
	Then the posterior distribution with weighted likelihood $\Pi_N^\pi$ converges to a sequence of normal distributions in the total variational distance:
	\begin{align*}
	\sup_{B}\bigabs{\Pi_N^\pi \big( \sqrt{N}(\theta-\theta_0) \in B\big\lvert D^{(N)}\big)-\mathcal{N}_{\Delta_{N,\theta_0}, I_{\theta_0}^{-1}}(B) } = \mathfrak{o}_{\mathbf{P}}(1).
	\end{align*}
\end{proposition}

\begin{lemma}\label{lem:LAN}
	Suppose that (A1) and (A2-CLT) holds, and that:
	\begin{enumerate}
		\item the map $\theta \mapsto \log p_{\theta}(x)=\ell_\theta(x)$ is differentiable at $\theta_0$ for all $x$ with derivative $\dot{\ell}_{\theta_0}(x)$, and for $\theta_1,\theta_2$ close enough to $\theta$,
		\begin{align*}
		\abs{\ell_{\theta_1}(x)-\ell_{\theta_2}(x)}\leq m(x)\pnorm{\theta_1-\theta_2}{}
		\end{align*}
		holds for some $P_{\theta_0}$-square integrable function $m$.
		\item The Kullback-Leibler divergence of $P_{\theta_0}$ is twice differentiable with a non-singular Hessian $I_{\theta_0}$: for $\theta$ close enough to $\theta_0$,
		\begin{align*}
		P_{\theta_0}\log \frac{p_{\theta_0}}{p_\theta} = \frac{1}{2} (\theta-\theta_0)I_{\theta_0} (\theta-\theta_0)+\mathfrak{o}\big(\pnorm{\theta-\theta_0}{}^2\big).
		\end{align*}
	\end{enumerate}
	Then the LAN condition in Proposition \ref{prop:bvm_generic} holds with $\Delta_{N,\theta_0}= I_{\theta_0}^{-1} \G_N^\pi \dot{\ell}_{\theta_0}$.
\end{lemma}
\begin{proof}
	Using Lemma 19.31 of \cite{van2000asymptotic}, we may conclude that the empirical process $
	\big\{\G_N\big(\sqrt{N}(\ell_{\theta_0+h/\sqrt{N}}-\ell_{\theta_0} )-h^\top \dot{\ell}_{\theta_0}\big)\equiv \G_N(f_h): \pnorm{h}{}\leq 1 \big\}$ converges weakly to $0$ in $\ell^\infty(h:\pnorm{h}{}\leq 1)$. Using the same arguments as in Proposition \ref{prop:multiplier_ineq} (but now in the probability form), it follows that for any $\epsilon>0$,
	\begin{align*}
	&\Prob\bigg(\sup_{\pnorm{h}{}\leq 1} \biggabs{\frac{1}{\sqrt{N}}\sum_{i=1}^N \frac{\xi_i}{\pi_i} \big(f_h(Y_i)-P_{\theta_0} f_h\big)  }>\epsilon \bigg)\lesssim \Prob\bigg(\sup_{\pnorm{h}{}\leq 1} \bigabs{ \G_N\big(f_h\big)}>\epsilon/C \bigg)\to 0.
	\end{align*}
	Hence by (A2-CLT), and the fact that $\sup_{\pnorm{h}{}\leq 1}\abs{P_{\theta_0} f_h} \leq \sup_{\pnorm{h}{}\leq 1} \sqrt{P_{\theta_0} f_h^2} \to 0$, we have
	\begin{align*}
	&\sup_{\pnorm{h}{}\leq 1} \bigabs{ \G_N^\pi\big(f_h\big)}\\
	&\leq \sup_{\pnorm{h}{}\leq 1} \biggabs{\frac{1}{\sqrt{N}}\sum_{i=1}^N \frac{\xi_i}{\pi_i} \big(f_h(Y_i)-P_{\theta_0} f_h\big)  }+ \sup_{\pnorm{h}{}\leq 1}\abs{P_{\theta_0} f_h}  \biggabs{\frac{1}{\sqrt{N}}\sum_{i=1}^N \bigg(\frac{\xi_i}{\pi_i}-1\bigg)}\\
	&= \mathfrak{o}_{\mathbf{P}}(1).
	\end{align*}
	This means that
	\begin{align*}
	\sup_{ \pnorm{h}{}\leq 1} \biggabs{N\Prob_N^\pi \log \frac{p_{\theta_0+h/\sqrt{N}}}{p_{\theta_0}}-\G_N^\pi h^\top \dot{\ell}_{\theta_0}-N\cdot P_{\theta_0}  \log \frac{p_{\theta_0+h/\sqrt{N}}}{p_{\theta_0}} } = \mathfrak{o}_{\mathbf{P}}(1).
	\end{align*}
	Using condition (2), we have
	\begin{align*}
	\sup_{\pnorm{h}{}\leq 1} \biggabs{ N\Prob_N^\pi \log \frac{p_{\theta_0+h/\sqrt{N}} }{p_{\theta_0}}-\G_N^\pi h^\top \dot{\ell}_{\theta_0}-\frac{1}{2}h^\top I_{\theta_0}h}=\mathfrak{o}_{\mathbf{P}}(1),
	\end{align*}
	proving the claim of the lemma.
\end{proof}

\begin{lemma}\label{lem:post_contraction}
	Suppose (A1) holds. Further assume that the local Gaussianity condition and the prior mass condition (2) in Theorem \ref{thm:bvm} hold. Then the posterior distribution with weighted likelihood $\Pi_N^\pi$ contracts at an $\sqrt{N}$-rate.
\end{lemma}
\begin{proof}
We will apply Theorem \ref{thm:bayes} with $d=\pnorm{\cdot}{}$. By a standard local entropy estimate for the finite-dimensional Euclidean space, we may take $\delta_N\equiv L_N/\sqrt{N}$ for any $L_N \to \infty$. For the prior mass condition, note that $\Pi$ has Lebesgue density bounded away from both $0$ and $\infty$ on $\Theta$, and hence for any $j \in \N$, with $\Theta_j\equiv \{\theta: j\delta_N<\pnorm{\theta-\theta_0}{}\leq (j+1) \delta_N\}$, 
	\begin{align*}
	 \frac{\Pi(\Theta_j)}{\Pi\big(\pnorm{\theta-\theta_0}{}\leq \delta_N\big)}\leq C_1 j^d\leq \exp(C_2j^2 N\delta_N^2)
	\end{align*}
holds with a small enough $C_2>0$ as long as $N$ is large enough. The conditions of Theorem \ref{thm:bayes} are now verified, and hence a $\sqrt{N}$-contraction rate is established.
\end{proof}

\begin{proof}[Proof of Theorem \ref{thm:bvm}]
	The claim follows by using Lemmas \ref{lem:LAN} and \ref{lem:post_contraction} in Proposition \ref{prop:bvm_generic}.
\end{proof}

\section{Ancillary results}\label{section:remaining_proof}

%We need to the following multiplier inequality due to \cite{han2017sharp}.
%
%\begin{proposition}[Theorem 1 of \cite{han2017sharp}]\label{prop:local_maximal_generic}
%	Suppose $\{M_i\}$ are independent of i.i.d. random variables $\{Y_i\}$. Let $\mathcal{F}_1\supset \cdots \supset \mathcal{F}_N$ be a non-increasing sequence of centered function classes. Assume further that there exist non-decreasing concave functions $\{\psi_N\}:\R_{\geq 0}\to \R_{\geq 0}$ with $\psi_N(0)=0$ such that
%	\begin{align}\label{cond:local_maximal_mulep}
%	\E \biggpnorm{\sum_{i=1}^k\epsilon_i f(X_i)}{\mathcal{F}_k}\leq \psi_N(k)
%	\end{align}
%	holds for all $1\leq k\leq N$. Then
%	\begin{align}
%	\E \biggpnorm{\sum_{i=1}^N M_i f(X_i)}{\mathcal{F}_N}\leq 4\int_0^\infty  \psi_N\bigg(\sum_{i=1}^{N} \Prob(\abs{M_i}>t)\bigg) \ \d{t}.
%	\end{align}
%\end{proposition}
%
%%The key ingredient in the proofs of Theorems \ref{thm:ULLN_HT} and \ref{thm:weak_convergence_HT} is 
%The following multiplier inequality that gives a $p$-th moment control of the following variant of the Horvitz-Thompson empirical process is useful.

\begin{proposition}\label{prop:multiplier_ineq}
	Suppose (A1) holds. Then for any countable class $\mathcal{F}$ and $p\geq 1$,
	\begin{align*}
	\E \biggpnorm{ \sum_{i=1}^N \frac{\xi_{i}}{\pi_i} \big(f(Y_i)-Pf\big) }{\mathcal{F}}^p \leq  (C/\pi_0)^p \cdot \E \biggpnorm{ \sum_{i=1}^N \big(f(Y_i)-Pf\big) }{\mathcal{F}}^p.
	\end{align*}
	Here $C>0$ is an absolute constant.
\end{proposition}
\begin{proof}[Proof of Proposition \ref{prop:multiplier_ineq}]
	%Let $\mathcal{G}_N\equiv \sigma(\{\xi_1\geq \ldots\geq \xi_N\})$. Then
	The proof is essentially contained in the proof of Theorem 1 of \cite{han2017sharp}. We only sketch some details. Denote $\eta_i\equiv \xi_i/\pi_i$, and let $\eta_{(1)}\geq \eta_{(2)}\geq \ldots\geq \eta_{(N)}\geq \eta_{(N+1)}=0$ be the reversed order statistics of $\{\eta_i\}_{i=1}^N$. Then, using the same arguments as in \cite{han2017sharp}, we have
	\begin{align*}
	\E \biggpnorm{ \sum_{i=1}^N \eta_{i} (f(Y_i)-Pf) }{\mathcal{F}}^p 
%	=\E\biggpnorm{ \sum_{i=1}^N \eta_{\sigma(i)} (f(Y_{\sigma(i)})-Pf) }{\mathcal{F}}^p \\
%	& =\E\biggpnorm{ \sum_{i=1}^N \sum_{k\geq i} (\eta_{(k)}-\eta_{(k+1)}) (f(Y_{\sigma(i)})-Pf) }{\mathcal{F}}^p\\
%	& =\E\biggpnorm{\sum_{k=1}^N (\eta_{(k)}-\eta_{(k+1)}) \sum_{i=1}^k(f(Y_{\sigma(i)})-Pf) }{\mathcal{F}}^p\\
	&\leq \E\bigg[  \biggabs{\sum_{k=1}^N (\eta_{(k)}-\eta_{(k+1)})}^p  \max_{1\leq k\leq N}\biggpnorm{\sum_{i=1}^k(f(Y_{i})-Pf) }{\mathcal{F}}^p\bigg]\\
	%&\leq (1/\pi_0)^p\cdot \E_{(\eta^{(N)},Z^{(N)}) } \max_{1\leq k\leq N}\bigg[ \E_{(Y^{(N)},Z^{(N)})|\eta^{(N)}}\biggpnorm{ \sum_{i=1}^k (f(Y_{i})-Pf) }{\mathcal{F}}^p\bigg] \\
%	& \leq  (1/\pi_0)^p \E \max_{1\leq k\leq N} \biggpnorm{ \sum_{i=1}^k (f(Y_{\sigma(i)})-Pf) }{\mathcal{F}}^p\\
	&\leq (1/\pi_0)^p \cdot \E\max_{1\leq k\leq N}\biggpnorm{ \sum_{i=1}^k (f(Y_{i})-Pf) }{\mathcal{F}}^p\\
	&\leq  (C/\pi_0)^p \cdot \E\biggpnorm{ \sum_{i=1}^N (f(Y_{i})-Pf) }{\mathcal{F}}^p.
	\end{align*}
	The last line follows from L\'evy-type maximal inequality (cf. Theorem 1.1.5 of \cite{de2012decoupling}):
	\begin{align*}
	&\E\max_{1\leq k\leq N}\biggpnorm{ \sum_{i=1}^k (f(Y_{i})-Pf) }{\mathcal{F}}^p  = \int_0^\infty \Prob\bigg(\max_{1\leq k\leq N}\biggpnorm{ \sum_{i=1}^k (f(Y_{i})-Pf) }{\mathcal{F}}>t\bigg) p t^{p-1}\ \d{t}\\
	&\leq 9 \int_0^\infty \Prob\bigg(\biggpnorm{ \sum_{i=1}^N (f(Y_{i})-Pf) }{\mathcal{F}}>t/30\bigg) p t^{p-1}\ \d{t}\\
	&\leq C^p \cdot \E\biggpnorm{ \sum_{i=1}^N (f(Y_{i})-Pf) }{\mathcal{F}}^p,
	\end{align*}
	as desired.
\end{proof}

%The key ingredient of proving Proposition \ref{prop:deviation_normalized_process} is the following.
The following is an analogue of the one-sided Talagrand's concentration inequality in the context of Horvitz-Thompson empirical processes.

\begin{proposition}\label{prop:Talagrand_HT}
	Suppose (A1) holds. Let $\mathcal{F}$ be a countable class of real-valued measurable functions such that $\sup_{f \in \mathcal{F}} \pnorm{f}{\infty}\leq b$. Then there exists some constant $C=C(\pi_0)>0$ such that for any $x\geq 0$,
	\begin{align*}
	&\Prob\bigg( C^{-1}\sup_{f \in \mathcal{F}}\biggabs{\sum_{i=1}^N \frac{\xi_i}{\pi_i} \big(f(Y_i)-Pf\big) }\\
	&\qquad\qquad \geq \E \sup_{f \in \mathcal{F}}\biggabs{\sum_{i=1}^N \big(f(Y_i)-Pf\big)} +\sqrt{N \sigma^2 x} + b x \bigg)\leq e\cdot e^{-x},
	\end{align*}
	where $\sigma^2\equiv \sup_{f \in \mathcal{F}} \mathrm{Var}_P f$.
\end{proposition}

One notable feature of the above Talagrand-type inequality is that we only need to compute the size of the empirical process $\E \sup_{f \in \mathcal{F}}\abs{\sum_{i=1}^N \big(f(Y_i)-Pf\big)}$ instead of the Horvitz-Thompson empirical process.

To prove Proposition \ref{prop:Talagrand_HT}, we need Talagrand's concentration inequality \cite{talagrand1996new} for the usual empirical process. 

\begin{lemma}\label{lem:talagrand_conc_ineq}
	Let $X_1,\ldots,X_N$ be i.i.d. with law $P$ on $(\mathcal{X},\mathcal{A})$. Let $\mathcal{F}$ be a countable class of $P$-centered real-valued measurable functions such that $\sup_{f \in \mathcal{F}} \pnorm{f}{\infty}\leq b$. Let $S_j\equiv \sup_{f \in \mathcal{F}} \abs{\sum_{i=1}^j f(X_i)}$. Then
	\begin{align*}
	\Prob\bigg(\max_{1\leq j\leq N}S_j \geq \E S_N +\sqrt{2\bar{\sigma}^2 x}+b x/3 \bigg)\leq e^{-x},
	\end{align*}
	where $\bar{\sigma}^2\equiv N\sigma^2+2b  \E S_N$ with $\sigma^2\equiv \sup_{f \in \mathcal{F}} \mathrm{Var}_P f$. Consequently, 
	\begin{align*}
	\E S_N^p \leq C_0^p\big((\E S_N)^p + p^{p/2} (N\sigma^2)^{p/2}+p^p b^p \big).
	\end{align*}
\end{lemma}
\begin{proof}
	The exponential inequality follows from Theorem 3.3.9 of \cite{gine2015mathematical}, and naturally translates to the following form: for some absolute constant $C>0$,
	\begin{align*}
	\Prob\bigg(\big(S_N-C\cdot \E S_N\big)_+ \geq x \bigg)\leq C\exp\left(-\frac{x^2}{C(N\sigma^2+ bx)}\right).
	\end{align*}
	Hence,
	\begin{align*}
	&\E \big(S_N-C\cdot \E S_N\big)_+^p \\
	&\leq C_1 p \left(\int_0^\infty x^{p-1} e^{-x^2/(C_1 N\sigma^2)}\ \d{x}+\int_0^\infty x^{p-1} e^{-x/(C_1 b)}\ \d{x}\right)\\
	%&\leq C_1 p \left( (C_1n\sigma^2)^{p/2} \int_0^\infty x^{p-1} e^{-x^2}\ \d{x} +(C_1 b)^p \int_0^\infty x^{p-1} e^{-x}\ \d{x}\right)\\
	&\leq C_2^p \left( \Gamma\big(p/2\big)(N\sigma^2)^{p/2}+ \Gamma(p) b^p \right)\leq C_3^p \left( p^{p/2} (N\sigma^2)^{p/2}+p^p b^p\right),
	\end{align*}
	which implies the desired moment inequality. Here $C_0,C_1,C_2,C_3>0$ are absolute constants. 
\end{proof}

We also need the following lemma that translates the moment inequality back to an exponential inequality.
\begin{lemma}
	\label{lem:moment_tail}
	If $Y$ is a non-negative random variable such that 
	\begin{align*}
	(\mathbb{E} Y^p)^{1/p} \leq A_1p + A_2p^{1/2}+A_3
	\end{align*}
	for all $p \in [1,\infty)$ and some $A_1,A_2 > 0$, $A_3 \geq 0$, then we have the following exponential bound: for every $t \geq 0$,
	\begin{align*}
	\mathbb{P}(Y\geq t + eA_3) \leq e\cdot  \exp\bigg(-\frac{t}{2eA_1}\wedge \frac{t^2}{4e^2A_2^2}\bigg).
	\end{align*}
\end{lemma}
\begin{proof}
	The proof is standard. We include some details for readers' convenience. Let $s\equiv \frac{t}{2eA_1}\wedge \frac{t^2}{(2eA_2)^2}$. For values of $t$ such that $s\geq 1$, we have by Markov's inequality that
	\[
	\mathbb{P}(Y\geq t+eA_3) \leq \biggl(\frac{A_1s+ A_2s^{1/2}+A_3}{t+eA_3}\biggr)^s\leq e^{-s}\leq e^{1-s}.
	\]
	For values of $t$ such that $s<1$, we trivially have $\mathbb{P}(Y\geq t + eA_3) \leq \mathbb{P}(Y\geq t)\leq e^{1-s}$, as desired.
\end{proof}

\begin{proof}[Proof of Proposition \ref{prop:Talagrand_HT}]
	Fix $p\geq 1$. By Proposition \ref{prop:multiplier_ineq} and Talagrand's concentration inequality (cf. Lemma \ref{lem:talagrand_conc_ineq}), we have
	\begin{align*}
	&\E \sup_{f \in \mathcal{F}} \biggabs{\sum_{i=1}^N \frac{\xi_i}{\pi_i} \big(f(Y_i)-Pf\big) }^p\leq (C/\pi_0)^p \cdot \E \sup_{f \in \mathcal{F}} \biggabs{\sum_{i=1}^N \big(f(Y_i)-Pf\big)}^p\\
	&\leq (C'/\pi_0)^p \bigg[\bigg(\E \sup_{f \in \mathcal{F}} \biggabs{\sum_{i=1}^N \big(f(Y_i)-Pf\big)}\bigg)^p + p^{p/2}(N\sigma)^{p/2}+p^p b^p\bigg],
	\end{align*}
	The claim now follows from Lemma \ref{lem:moment_tail}.
\end{proof}

Let $\phi$ be a continuous and strictly increasing function with $\phi(0)=0$. Let $\mathcal{F}(r)\equiv \{f \in \mathcal{F}: \sigma_P^2 f\leq r^2\}$ and $\mathcal{F}(r,s] \equiv \mathcal{F}(s)\setminus \mathcal{F}(r)$. Fix $0<r<\delta\leq 1$. For a real number $1<q\leq 2$, let $l\equiv l_{r,\delta,q}$ be the smallest integer no smaller than $\log_q(\delta/r)$. Let for any $\bm{s}\equiv (s_1,\ldots,s_l) \in \R_{\geq 0}^l$, 
\begin{align}\label{def:beta_tau}
\beta_{N,q}(r,\delta)&\equiv \max_{1\leq j\leq l} \frac{\E \pnorm{\G_N}{\mathcal{F}(rq^{j-1},rq^j)} }{\phi(rq^j)},\quad\tau_{N,q}(r,\delta,\bm{s})\equiv \max_{1\leq j\leq l} \frac{rq^j \sqrt{ s_j}+s_j/\sqrt{N}}{\phi(rq^j)}.
\end{align}

\begin{proposition}\label{prop:deviation_normalized_process}
	Suppose (A1) holds. Assume that $\phi$ is continuous, strictly increasing and satisfies $\sup_{r\leq x\leq 1} \phi(qx)/\phi(x)=\kappa_{r,q}<\infty$ for some $1<q\leq 2$. Then for any $\bm{s}\equiv (s_1,\ldots,s_l) \in \R_{\geq 0}^l$, 
	\begin{align*}
	&\Prob\left[\sup_{f \in \mathcal{F}: r^2<\sigma_P^2 f\leq \delta^2} \frac{\abs{\tilde{\G}^\pi_N(f)}}{\phi\left(\sigma_P f\right)}\geq C\kappa_{r,q}\bigg(\beta_{N,q}(r,\delta)+ \tau_{N,q}(r,\delta,\bm{s})\bigg)\right]\leq 	e\sum_{j=1}^l \exp\big(-s_j\big).
	\end{align*}
	Here $C>0$ is a constant depending only through $\pi_0>0$.
\end{proposition}

\begin{proof}[Proof of Proposition \ref{prop:deviation_normalized_process}]
	The proof is a simple application of the one-sided Talagrand's concentration inequality for the Horvitz-Thompson empirical process (cf. Proposition \ref{prop:Talagrand_HT}) combined with a peeling device, analogous to the developement in \cite{gine2006concentration}. Write $\mathcal{F}_j\equiv \mathcal{F}(rq^{j-1},rq^j]$ and $\phi_q(u)\equiv \phi(rq^j)$ if $u \in (rq^{j-1},rq^j]$ for notational convenience. By Proposition \ref{prop:Talagrand_HT},
	\begin{align*}
	\Prob\bigg[\pnorm{\tilde{\G}_N^\pi}{\mathcal{F}_j}\geq C\left(\E \pnorm{\G_N}{\mathcal{F}_j} + \sqrt{\sigma^2_j s_j}+\frac{s_j}{\sqrt{N}}\right)\bigg]\leq e\cdot \exp\big(-s_j\big)
	\end{align*}
	where $\sigma_j^2 = \sup_{f \in \mathcal{F}_j} \sigma_P^2f=r^2q^{2j}$. Hence by a union bound we see that with probability at least $1-e\sum_{j=1}^l  \exp(-s_j)$, it holds that
	\begin{align*}
	&\bigg(\sup_{f \in \mathcal{F}: r^2<\sigma_P^2 f\leq \delta^2} \frac{\abs{\tilde{\G}_N^\pi(f)}}{\phi_q\left(\sigma_P f\right)}-C\beta_{N,q}(r,\delta)\bigg)_+\\
	&\leq \max_{1\leq j\leq l}\left( \frac{ \pnorm{\G_N }{\mathcal{F}_j} }{\phi(rq^j)}- \frac{C\E \pnorm{\G_N}{\mathcal{F}(rq^{j-1},rq^j)} }{\phi(rq^j)}\right)_+\leq C\max_{1\leq j\leq l} \frac{rq^j \sqrt{ s_j}+s_j/\sqrt{N}}{\phi(rq^j)}.
	\end{align*}
	Now the conclusion follows from $\sup_{r\leq x\leq 1}\phi(qx)/\phi(x)=\kappa_{r,q}<\infty$.
\end{proof}

Let
\beqa\label{def:uniform_entropy}
J(\delta,\mathcal{F},L_2) \equiv   \int_0^\delta  \sup_Q\sqrt{1+\log \mathcal{N}(\epsilon\pnorm{F}{Q,2},\mathcal{F},L_2(Q))}\ \d{\epsilon}
\eeqa
denote the \emph{uniform} entropy integral, where the supremum is taken over all finitely discrete probability measures, and let
\begin{align}\label{def:bracketing_entropy}
J_{[\,]}(\delta,\mathcal{F},\pnorm{\cdot}{}) \equiv \int_0^\delta \sqrt{1+\log \mathcal{N}_{[\,]}(\epsilon,\mathcal{F},\pnorm{\cdot}{})}\ \d{\epsilon}
\end{align}
denote the \emph{bracketing} entropy integral. 
%The following local maximal inequalities for the empirical process play a key role throughout the proof.
\begin{lemma}\label{lem:local_maximal_ineq}
	Suppose that $\mathcal{F} \subset L_\infty(1)$, and $X_1,\ldots,X_n$'s are i.i.d. random variables with law $P$. Then with $\mathcal{F}(\delta)\equiv \{f \in \mathcal{F}:Pf^2<\delta^2\}$,
	\begin{enumerate}
		\item If the uniform entropy integral (\ref{def:uniform_entropy}) converges, then
		\begin{align}\label{ineq:local_maximal_uniform}
		\E \biggpnorm{\sum_{i=1}^n \epsilon_i f(X_i)}{\mathcal{F}(\delta)}  \lesssim \sqrt{n}J(\delta,\mathcal{F},L_2)\bigg(1+\frac{J(\delta,\mathcal{F},L_2)}{\sqrt{n} \delta^2 \pnorm{F}{P,2}}\bigg)\pnorm{F}{P,2}.
		\end{align}
		\item If the bracketing entropy integral (\ref{def:bracketing_entropy}) converges, then
		\begin{align}\label{ineq:local_maximal_bracketing}
		\E \biggpnorm{\sum_{i=1}^n \epsilon_i f(X_i)}{\mathcal{F}(\delta)} \lesssim \sqrt{n} J_{[\,]}(\delta,\mathcal{F},L_2(P))\bigg(1+\frac{J_{[\,]}(\delta,\mathcal{F},L_2(P))}{\sqrt{n} \delta^2}\bigg).
		\end{align}
	\end{enumerate}
\end{lemma}

\begin{proof}
	(\ref{ineq:local_maximal_uniform}) follows from \cite{van2011local}; see also Section 3 of \cite{gine2006concentration}, or Theorem 3.5.4 of \cite{gine2015mathematical}. (\ref{ineq:local_maximal_bracketing}) follows from Lemma 3.4.2 of \cite{van1996weak}.
\end{proof}

\begin{lemma}\label{lem:sum_cond_CLT}
	Let $\{U_N\}$ be a sequence of random variables defined on $(\mathcal{S}_N\times\mathcal{X},\sigma(\mathcal{S}_N)\times\mathcal{A},\Prob)$ such that $U_N\rightsquigarrow \mathcal{N}(0,\tau^2)$ under $\Prob_d(\cdot,\omega)$ for $\Prob_{(Y,Z)}$-a.s. $\omega \in \mathcal{X}$. Let $\{V_N\}$ be another sequence of random variables defined on $(\mathcal{X},\mathcal{A},\Prob_{(Y,Z)})$ such that $V_N\rightsquigarrow \mathcal{N}(0,\sigma^2)$ under $\Prob_{(Y,Z)}$. Then $U_N+V_N\rightsquigarrow \mathcal{N}(0,\tau^2+\sigma^2)$ under $\Prob$.
\end{lemma}
\begin{proof}[Proof of Lemma \ref{lem:sum_cond_CLT}]
	Consider the characteristic function: for any $t \in \R$, we have
	\begin{align*}
	&\abs{\E e^{it(U_N+V_N)}-e^{it(\tau^2+\sigma^2)}}\\
	&\leq \abs{\E e^{it(U_N+V_N)}-e^{it\tau^2} \E e^{itV_N}}+\abs{e^{it\tau^2} \E e^{itV_N}-e^{it\tau^2}e^{it\sigma^2}}\\
	& = \abs{ \E \big(\E[e^{it U_N}|(Y^{(N)},Z^{(N)})]-e^{it\tau^2}\big)\cdot e^{itV_N} \big)  }+ \abs{\E e^{it V_N}-e^{it\sigma^2}}.
	\end{align*}
	The first term in the above display vanishes as $N\to \infty$ by the conditional CLT assumption on $U_N$ and the dominated convergence theorem, while the second also vanishes by the CLT assumption on $V_N$.
\end{proof}

\section*{Acknowledgements}
The authors would like to thank Thomas Lumley for several helpful suggestions.

\bibliographystyle{amsalpha}
\bibliography{mybib}

\end{document}